\documentclass[11pt]{amsart}
\usepackage{enumerate}
\usepackage{amsmath}
\usepackage{amsthm}
\usepackage{amsfonts}
\usepackage{amssymb}
\usepackage[numbers]{natbib}
\usepackage{graphicx,xcolor}
\graphicspath{ {Diagrams/} }
\usepackage{graphicx}

\usepackage{fullpage}

\def\COMMENT#1{}
\let\COMMENT=\footnote

\newtheorem{question}{Question}
\newtheorem{corollary}[question]{Corollary}
\newtheorem{problem}[question]{Problem}
\newtheorem{conjecture}[question]{Conjecture}
\newtheorem{theorem}[question]{Theorem}

\newtheorem{proposition}[question]{Proposition}
\newtheorem{lemma}[question]{Lemma}
\newtheorem{remark}[question]{Remark}
\newtheorem{claim}[question]{Claim}
\newtheorem{definition}[question]{Definition}
\newtheorem{construction}[question]{Construction}
\numberwithin{question}{section}
\numberwithin{equation}{section}

\begin{document}
\title{Extremal problems for multigraphs}
\author{A. Nicholas Day, Victor Falgas-Ravry and
 Andrew Treglown}

\thanks{ND: Ume{\aa} Universitet, Sweden\\
\indent VFR:  Ume{\aa} Universitet, Sweden, {\tt victor.falgas-ravry@umu.se}. Research supported by VR grant 2016-03488.\\
\indent AT: University of Birmingham, United Kingdom, {\tt a.c.treglown@bham.ac.uk}. }

\begin{abstract}
An $(n,s,q)$-graph is an $n$-vertex multigraph in which every $s$-set of vertices spans at most $q$ edges. Tur\'an-type questions on the maximum of the sum of the edge multiplicities in  such multigraphs have been studied since the 1990s. More recently, Mubayi and Terry [An extremal problem with a transcendental solution, Combinatorics Probability and Computing 2019] posed the problem of determining the maximum of the product of the edge multiplicities in
$(n,s,q)$-graphs. We give a general lower bound construction for this problem for many pairs $(s,q)$, which we conjecture is asymptotically best possible. We prove various general cases of our conjecture, and in particular we settle a conjecture of Mubayi and Terry on the $(s,q)=(4,6a+3)$ case  of the problem (for $a\geq2$); this in turn answers a question of Alon. We also determine the asymptotic behaviour of the problem for `sparse' multigraphs (i.e. when $q\leq 2\binom{s}{2}$).
Finally we introduce some tools that are likely to be useful for attacking the problem in general.

\end{abstract}
\maketitle

\section{Introduction}\label{section: introduction}
In 1963, Erd\H{o}s~\cite{erdos1, erdos2} raised the question of determining $\text{ex} (n,s,q)$, the maximum number
 of edges in an $n$-vertex graph in which every $s$-set of vertices spans at most $q$ edges, for some integer $q$ where $0 \leq q \leq \binom{s}{2}$. This can be viewed as asking for an extension of Tur\'an's theorem (which covers the case $q=\binom{s}{2}-1$).
In the range $q\geq  \lfloor s^2/4\rfloor$, $\text{ex} (n,s,q)$ is quadratic in $n$, and the Erd\H{o}s--Stone--Simonovits theorem~\cite{es1, es2} provides an asymptotically exact solution to this problem.  It is also known that $\text{ex} (n,s,q)$ is linear for $q<s-1$, and grows at a superlinear but subquadratic rate in the range $s-1\leq q <\lfloor s^2/4\rfloor$. As Erd{\H o}s~\cite{erdos1} pointed out, determining $\mathrm{ex}(n, s, s-1)$ is equivalent to determining the Tur\'an number of the cycle of length $2\lfloor s/2\rfloor$, which is itself essentially equivalent to determining the maximum number of edges $\mathrm{ex}(n, \{C_3,C_4, \ldots, C_{2\lfloor s/2\rfloor}\})$ in an $n$-vertex graph of girth at least $2\lfloor s/2\rfloor+1$. We refer an interested reader to~\cite[Section 3]{fk} for an overview of some of the results and rich history relating to this question of Erd{\H o}s on the value of $\text{ex} (n,s,q)$.

Since the late 1990s there has been an interest in studying the analogous problem in the weighted or multigraph setting. A \emph{multigraph} is a pair $(V,w)$, where $V$ is a set of vertices and $w$ is a function
$w: \binom{V}{2} \rightarrow \mathbb Z _{\geq 0}$.
\begin{definition}
Given integers $s\geq 2 $ and $q \geq 0$,
we say a multigraph $G=(V,w)$ is an \emph{$(s,q)$-graph} if every $s$-set of vertices in $V$ spans at most $q$ edges; i.e. $\sum _{xy \in \binom{X}{2}} w(xy) \leq q$ for every $X \in \binom{V}{s}$.
An \emph{$(n,s,q)$-graph} is an $n$-vertex $(s,q)$-graph. We write $\mathcal F(n,s,q)$
 for the set of all
$(n,s,q)$-graphs with vertex set $[n]:=\{1, \dots, n\}$.
\end{definition}
The initial work of Bondy and Tuza~\cite{bondy}, Kuchenbrod~\cite{Kuchenbrod} and F\"uredi and K\"undgen~\cite{fk} on $(n,s,q)$-graphs focussed on determining the maximum of the sum of the edge multiplicities in an $(n,s,q)$-graph. In particular, F\"uredi and K\"undgen proved the existence of an explicit constant $m=m(s,q)$ such that every $(n,s,q)$-graph has at most $m\binom{n}{2}+O(n)$ edges, a result which they show is asymptotically tight~\cite[Theorem 1]{fk}. Thus the `sum-version' of Erd{\H o}s's problem for multigraphs is asymptotically understood. 
More recently, Mubayi and Terry~\cite{mt1, mt2} introduced a `product version' of the problem, which we describe below.
\begin{definition}
Given a  multigraph $G=(V,w)$, we define\footnote{The existence of the limit $\text{ex}_{\Pi}(s,q)$ below follows from an easy averaging argument showing  $\text{ex} _\Pi (n,s,q)^{1/\binom{n}{2}}$ is nonincreasing in $n$ for $n\geq s$ and bounded below by $1$.}
\begin{eqnarray*}
P(G)&:=& \prod _{xy \in \binom{V}{2}} w(xy), \\
\text{ex} _\Pi (n,s,q)&:=& \max \{ P(G) : G \in \mathcal{F}(n,s,q)\}, \\
\text{ex} _\Pi (s,q) &:=& \lim_{n \rightarrow \infty}\left( \text{ex} _\Pi \left(n,s,q \right)\right)^{{n \choose 2}^{-1}}. \\
\end{eqnarray*}
\end{definition}
Mubayi and Terry's motivation for studying the quantity $\text{ex} _\Pi (n,s,q)$ is connected to container theory and attempts to develop multigraph versions of the counting theorems of Erd{\H o}s--Kleitman--Rothschild. Explicitly, using the powerful hypergraph container theories developed by Balogh, Morris and Samotij~\cite{BaloghMorrisSamotij} and Saxton--Thomason~\cite{SaxtonThomason}, Mubayi and Terry showed in~\cite[Theorem 2.2]{mt1} that for $q>\binom{s}{2}$,
\begin{align}\label{eq: MT counting}
\Bigl\vert \mathcal{F}\bigl(n,s,q-\binom{s}{2}\bigr)\Bigr\vert = \text{ex} _\Pi (s,q)^{\binom{n}{2}+o(n^2)}.
\end{align}
In particular, solving the Erd{\H o}s--Kleitman--Rothschild-type counting problem of estimating the size of the multigraph family $\mathcal{F}(n,s,q-\binom{s}{2})$ is equivalent to the Tur\'an-type extremal problem of determining $\text{ex} _\Pi (n,s,q)$. Mubayi and Terry thus raised the general problem of determining $\text{ex} _\Pi (n,s,q)$ and initiated its study in two recent papers~\cite{mt1, mt2}.
\begin{problem}[Mubayi--Terry multigraph problem]\label{problem: Mubayi--Terry}
Given positive integers $s\geq 2$ and $q$, determine $\mathrm{ex}_{\Pi}(n,s,q)$.
\end{problem}	
The main result of Mubayi and Terry in~\cite{mt1} was a proof that
$$\text{ex} _\Pi (n,4,15)=2^{\gamma n^2 +O(n)}$$
where $\gamma$ is 
defined by
$$\gamma := \frac{\beta ^2}{2} +\beta (1-\beta )\frac{\log 3}{\log 2} \ \text{ where } \ \beta :=
\frac{\log 3}{2\log 3- \log 2}.$$
Assuming Schanuel's conjecture from number theory, both $\gamma$ and $2^{\gamma}$ are transcendental (see~\cite[Proposition 2.6]{mt1}). According to Mubayi and Terry~\cite{mt1}, this is the first example of a `fairly natural extremal graph problem' whose
asymptotic answer is given by an explicitly defined transcendental number.
In response to a question of Alon on whether this transcendental behaviour is an isolated case,
Mubayi and Terry made a conjecture~\cite[Conjecture 6.3]{mt1}  which, if true, would provide infinitely many examples
of such behaviour.  In this paper we resolve their conjecture fully (see Theorem~\ref{theorem: mubayi Terry conjecture} below), in turn resolving the question of Alon.

In~\cite{mt2} Mubayi and Terry determined $\text{ex}_\Pi (n,s,q)$ exactly or asymptotically for  
pairs $(s,q)$ where $a\binom{s}{2}-\frac{s}{2}\leq q\leq a\binom{s}{2}+s-2$ for some  $a\in \mathbb N$. Building on this work we provide solutions to the problem for a further range of values of $(s,q)$;  some of our results extend  those given in~\cite{mt2}, while others are new.  Further we give general lower bounds on the value of $\text{ex}_{\Pi}(n, s,q)$, which are entirely new and which may be one of the main contributions of the paper.

Before we formally state our results in Section~\ref{subsection: results}, we first need to describe a class of multigraph constructions which may be seen as multigraph analogues of the well known Tur\'{a}n graphs.  This can be found in Section~\ref{subsection: lower bound construction}. We first present the notation and definitions we will use throughout this paper.

\section{Definitions and notation}\label{section: definitions and notation}
Given a set $A$ and $r\in \mathbb{Z}_{\geq 0}$, we let $A^{(r)}$ denote the collection of all subsets of $A$ of size $r$. A multigraph is a pair $G=(V,w)$, where $V=V(G)$ is a set of vertices and $w=w_G$ is a function $w: \ V^{(2)}\rightarrow \mathbb{Z}_{\geq 0}$ assigning to each pair $\{a,b\}\in V^{(2)}$ a \emph{weight} or \emph{multiplicity} $w_G(\{a,b\})$. We usually write $ab$ for $\{a,b\}$ and, when the host multigraph $G$ is clear from context, we omit the subscript $G$ and write simply $w(ab)$ for $w_G(\{a,b\})$. We write $v(G):=|V(G)|$.

Given a multigraph $G$ and a set $X\subseteq V(G)$, we write $e(G[X])$ or, when the host multigraph $G$ is clear from context, $e(X)$ for the sum of the edge multiplicities of $G$ inside $X$, i.e. \[e(G[X]):=\sum_{v_1v_2 \in X^{(2)}} w(v_1v_2).\] 
Similarly, we write $P(G[X])$ or $P(X)$ for the product of the edges multiplicities of $G$ inside $X$, i.e. 
\[P(G[X]):=\prod_{v_1v_2\in X^{(2)}}w(v_1v_2).\] 
Further, given disjoint sets $X,Y\subset V(G)$ we write $e(G[X,Y])$  ($e(X,Y)$) and $P(G[X,Y])$ ($P(X,Y)$) for, respectively the sum  and the product of the edge multiplicities of $G$ over all edges $xy$ with $x\in X$ and $y\in Y$. Given a vertex $v\in V(G)$, we write $p_G(v)$ for the product of the edge multiplicities over all edges of $G$ containing $v$; again, when $G$ is clear from context we omit the subscript and write simply $p(v)$. 
We refer to this quantity as the \emph{product-degree} of $v$ in $G$.
We define $d_G (v)$ (or simply $d(v)$) to be the sum of the edge multiplicities over all edges of $G$ containing $v$, and refer to this quantity as the \emph{degree}
of $v$ in $G$. For $n\in \mathbb{N}$ we write $[n]$ as a shorthand for the set $\{1,2,\ldots ,n\}$ and $[0,n]$ as shorthand for $\{0\} \cup [n]$.

\section{Lower bound construction,  statements of results, and a conjecture}\label{section:lower bound constructions, results and conjectures}
\subsection{A conjectural extremal construction}\label{subsection: lower bound construction}
A key goal of the paper is to introduce a natural class of constructions which we believe give extremal examples for many general cases of the Mubayi--Terry multigraph problem. 
\begin{construction}\label{construction: lower bound}
Let $a,r \in \mathbb{N}$ and $d\in [0,a-1]$. Given $n \in \mathbb{N}$, let $\mathcal{T}_{r,d}(a,n)$ denote the collection of multigraphs $G$ on $[n]$ for which $V(G)$ can be partitioned into $r$ parts $V_0, \ldots ,V_{r-1}$ such that:
\begin{enumerate}[(i)]
	\item all edges in $G[V_0]$ have multiplicity $a-d$;
	\item for all $i\in [r-1]$,  all edges in $G[V_i]$ have multiplicity $a$;
	\item all other edges of $G$ have multiplicity $a+1$.
\end{enumerate} 
Given $G\in \mathcal{T}_{r,d}(a,n)$, we refer to $\sqcup_{i=0}^{r-1}V_i$ as the \emph{canonical partition} of $G$.  
\end{construction}	
This class of constructions generalises constructions arising in the work of Mubayi and Terry~\cite{mt1, mt2} (these correspond to the special cases $d=0$ and $(r,d)=(2,1)$ in our construction). See Figure~\ref{figure: T-example} for an example of what these graphs look like when $r = 4$.
\begin{figure}[ht]
	\centering
\includegraphics[scale=0.8]{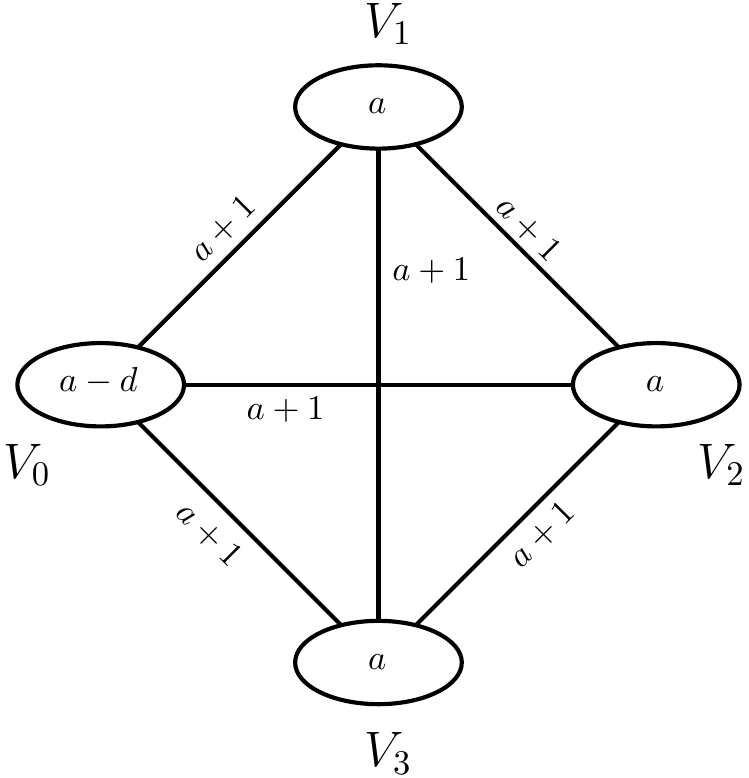}
	\caption{An example of the structure of graphs in $\mathcal {T}_{4,d}(a,n)$.}  
	\label{figure: T-example}
\end{figure}	
We write
\begin{equation}
\Sigma_{r,d}(a,n) := \max \{e(G) : G  \in \mathcal{T}_{r,d}(a,n)\}, \nonumber 
\end{equation}
and
\begin{equation}
\Pi_{r,d}(a,n) := \max \{P(G) : G  \in \mathcal{T}_{r,d}(a,n)\}. \nonumber
\end{equation}
Let $T^{e}_{r,d}(a,n)$ denote the family of multigraphs $G \in \mathcal{T}_{r,d}(a,n)$ with $e(G)=\Sigma_{r,d}(a,n)$ and $T^{P}_{r,d}(a,n)$ denote the family of multigraphs $G\in \mathcal{T}_{r,d}(a,n)$ with $P(G) =\Pi_{r,d}(a,n)$.  These two multigraph families are in general quite different. Indeed, let $\sqcup_{i=0}^{r-1}V_i$ be a canonical partition of $G\in \mathcal{T}_{r,d}(a,n)$, and set $\vert V_i\vert =v_in$. Given a multigraph	$G\in \mathcal{T}_{r,d}(a,n)$, we have $e(G)= (a+1)\binom{n}{2}  -(d+1) \binom{v_0n}{2}- \sum_{i=1}^{r-1} \binom{v_in}{2}$.  An easy exercise in optimisation shows that this is maximised when 
\begin{align}\label{eq: part sizes, sum-extremal inside Turan family}
v_0=\frac{1}{d(r-1)+r}+O(n^{-1}) && \textrm{ and } && v_i = \frac{1-v_0}{r-1}+O(n^{-1}) \textrm{ for }i\in [r-1].
\end{align}	
On the other hand,  
\begin{align}\label{eq1}
\log_{(a+1)}P(G)= \binom{n}{2} -  \left(\log_{(a+1)}\Bigl(\frac{a+1}{a-d}\Bigr)\right)\binom{v_0n}{2} - \left(\log_{(a+1)}\Bigl(\frac{a+1}{a}\Bigr)\right)\sum_{i=1}^{r-1}\binom{v_in}{2},
\end{align}
 which is maximised when
\begin{align}\label{eq: part sizes, product-extremal inside Turan family}
v_0=\frac{\log \left(\frac{a+1}{a}\right)}{\log \left(\frac{(a+1)^{r}}{a(a-d)^{r-1}}\right)}+O(n^{-1}) && \textrm{ and } && v_i = \frac{1-v_0}{r-1}+O(n^{-1}) \textrm{ for }i\in [r-1].
\end{align}	
Thus when $d=0$, having $v_0=1/r+O(n^{-1})$ and a balanced $r$-partition maximises both the edge-sum and edge-product of $G\in \mathcal{T}_{r,d}(a,n)$. For other $d$, however, the weightings $v_0, v_1, \ldots ,v_{r-1}$ in optimal partitions for edge-sum and edge-product are very different. In particular, while $v_0\approx 1/\left((d+1)(r-1)+1\right)$ if we want to maximise the edge-sum, it must be chosen strictly smaller if we want to maximise the edge-product.  Explicitly, given fixed $r\in \mathbb{N}$, we define a function $x_{r\star}=x_{r\star}(a,d)$ from the set of pairs $(a,d)$ with  $d\in \mathbb{Z}_{\geqslant 0}$ and $a\in \mathbb{Z}_{>d}$ by 
\begin{align}\label{eq: def of xstar}
x_{r\star}(a,d):= \frac{\log \left(\frac{a+1}{a}\right)}{\log \left(\frac{(a+1)^{r}}{a(a-d)^{r-1}}\right)}.
\end{align}
Thus the graphs in $\mathcal{T}_{r,d}(a,n)$ that maximise the edge-product will have $v_{0} = x_{r\star}(a,d) + O(n^{-1})$.  By considering the product-degrees of vertices in $\mathcal{T}_{r,d}(a,n)$, noting that they will have as equal product-degree as possible, and then letting $n$ increase to infinity, we equivalently have that $x_{r\star}(a,d)$ is the unique solution to
\begin{equation}\label{eq: x_star property}
(a-d)^{x}(a+1)^{1-x} = a^{\frac{1-x}{r-1}} (a+1)^{\frac{r-2+x}{r-1}}. 
\end{equation}
Here the left-hand side arises by considering a vertex from $V_0$, the right-hand side by considering a vertex outside of $V_0$. As we show in Proposition~\ref{prop: a-monotonicity of xstar} below, $x_{r\star}<1/\left((d+1)(r-1)+1\right)$, so that product-extremal and sum-extremal optimal partitions (and thus the associated multigraphs) differ significantly. In Section~\ref{sec4} we prove some further results about $x_{r\star}(a,d)$ and $\Sigma_{r,d}(a,n)$ that we will need to proceed.  In particular, in the extremal results in this paper, we will need some information about the smallest integer value of $s$ for which $\Sigma_{r,d}(a, s)<\Sigma_{r,d-1}(a,s)$ holds, which turns out to be $s=(r-1)(d+1)+2$ (see Corollary~\ref{corollary: why (r-1)(d+1)+2} below).


We are now ready to state our main conjecture, which will underpin much of the work we do and the results we prove throughout this paper.
\begin{conjecture}\label{conjecture: entropy densities} 
For all integers $a,r,s, d$ with $a,r \geq 1$, $d\in [0,a-1]$, $s\geq (r-1)(d+1)+2$ and all $n$ sufficiently large,
\[ \mathrm{ex}_{\Pi}(n, s, \Sigma_{r,d}(a,s)) = \Pi_{r,d}(a,n).\]
\end{conjecture}	
Thus, roughly speaking, the conjecture states that given a (not too small) $s$, and 
$q$ equal to the maximum  number of edges allowed in any $s$-set of vertices in a multigraph from $\mathcal{T}_{r,d}(a,n)$, it is a multigraph $G$ from $\mathcal{T}_{r,d}(a,n)$ that maximises the edge-product $P(G)$ amongst all $(n,s,q)$-graphs.

The reader might wonder  where the lower bound on $s$ in the conjecture comes from. In fact, when $s$ is smaller, the problem is covered by other cases in the conjecture:
given $r+1\leq s \leq (r-1)(d+1)+1$, Corollary~\ref{corollary: why (r-1)(d+1)+2} below
implies that there is a least integer $d'$ with $d>d' \geq 0$,
$\Sigma_{r,d}(a,s)=\Sigma_{r,d'}(a,s)$ and  $s> (r-1)(d'+1) +1$. 
Thus the conjecture states that $\text{ex}_{\Pi}(n,s, \Sigma_{r,d}(a,s))$ should be $\Pi_{r,d'}(a, n)$   (and not the smaller quantity $\Pi_{r,d}(a,n)$).\footnote{Note the remaining case when $s\leq r$ has a simple solution. Indeed, in this case
$\mathrm{ex}_{\Pi}(n, s, \Sigma_{r,d}(a,s))=\mathrm{ex}_{\Pi}(n, s, (a+1)\binom{s}{2})$ which as noted in~\cite{mt2} equals $(a+1)^{\binom{n}{2}}$.}

\smallskip

 Set $\pi_{r,d}(a):=\lim_{n\rightarrow \infty} \log \left(\Pi_{r,d}(a,n)\right)/\binom{n}{2}$.  Note that if true, Conjecture~\ref{conjecture: entropy densities} implies that $\mathrm{ex}_{\Pi}( s, \Sigma_{r,d}(a,s))= e^{\pi_{r,d}(a)} $ for all integers $a,r,s,d$ with $a,r\geq 1$, $0\leq d<a$ and $s\geq (r-1)(d+1)+2$.
 Using (\ref{eq1}) and (\ref{eq: part sizes, product-extremal inside Turan family}) the limiting quantity $\pi_{r,d}(a)$ is easily computed: 
\begin{align}\label{eq: pi(r,d)(a,n)}
\pi_{r,d}(a)= \log(a+1) - (x_{r\star})^2 \log \left(\frac{a+1}{a-d}\right) - \left(\frac{(1-x_{r\star})^2}{r-1}\right)\log \left(\frac{a+1}{a}\right),
\end{align}
where $x_{r\star}=x_{r\star}(a,d)$ is as in~\eqref{eq: def of xstar}. Following~\cite{FalgasRavryOConnellUzzell19}, we refer to $\pi_{r,d}(a)$ as the \emph{entropy density} of $\mathcal{T}_{r,d}(a,n)$.  For any $a$ fixed, the values of the entropy densities $\pi_{r,d}(a)$ follow the lexicographic order on pairs $(r,d)$:
\begin{proposition}\label{prop: ordering of the entropy densities}
Let $a\in \mathbb{N}$ and $r_1, r_2\in \mathbb{N}$, $d_1,d_2\in [0, a-1]$. Then $\pi_{r_1,d_1}(a) \geq \pi_{r_2, d_2}(a)$ if and only if either $r_1> r_2$ or $r_1=r_2$ and $d_1\leq d_2$.
\end{proposition}
\begin{proof}
This is a straightforward exercise in combinatorial optimisation.
\end{proof}
\begin{remark}\label{remark: transcendance}
	Assuming Schanuel's conjecture from number theory and using Mih\u{a}ilescu's Theorem (his proof of Catalan's conjecture~\cite{Mihailescu04}), one can show that for all integers $a>d>1$ and $r\geq 2$ the quantities $x_{r\star}(a,d)$, $\pi_{r,d}(a)$ and $e^{\pi_{r,d}(a)}$ are all transcendental numbers --- see Proposition~\ref{prop: transcendentality} in the appendix.
\end{remark}

\subsection{Our results}\label{subsection: results}
With Conjecture~\ref{conjecture: entropy densities} stated, we can now give our main results.
Mubayi and Terry ~\cite{mt2} determined the value of $\mathrm{ex}_{\Pi}(4, a\binom{4}{2} + c)$ when $c \in \{0,1,2,4,5\}$:
\begin{align*}
\mathrm{ex}_{\Pi}(4, a\binom{4}{2})=\mathrm{ex}_{\Pi}&(4, a\binom{4}{2}+1)=\mathrm{ex}_{\Pi}(4, a\binom{4}{2}+2)=e^{\pi_{1,0}(a)} =  a,\\
\mathrm{ex}_{\Pi}(4, a\binom{4}{2}+4)&=e^{\pi_{2,0}(a)}= \sqrt{a(a+1)}, \\
\mathrm{ex}_{\Pi}(4, a\binom{4}{2}+5)&=e^{\pi_{3,0}(a)}=  a^{\frac{1}{3}}(a+1)^{\frac{2}{3}}.
\end{align*}
Note that in the cases above where $\mathrm{ex}_{\Pi}(4, a\binom{4}{2} + c)=a$, an asymptotically tight construction is simply the multigraph on $[n]$ in which every edge has multiplicity $a$ (which is the unique member of $\mathcal{T}_{1,0}(n)$).

For the case $c = 3$, Mubayi and Terry~\cite[Conjecture~6.3]{mt1} conjectured that for all $a\geq 2$,
\begin{equation}
\mathrm{ex}_{\Pi}(4, a\binom{4}{2}+3)=e^{\pi_{2,1}(a)}
\end{equation}
and proved this holds true when $a = 2$.  One of the main results of the present paper is the following theorem,  which proves the conjecture of Mubayi and Terry by determining the exact value of $\mathrm{ex}_{\Pi}(n, 4, \binom{4}{2}a+3)$ for all $n\geq 30$.
\begin{theorem}[Mubayi--Terry Conjecture]\label{theorem: mubayi Terry conjecture}
	For all integers $a\geq 2$ and $n\geq 30$,
	\[\mathrm{ex}_{\Pi}(n, 4, \binom{4}{2}a +3) = \Pi_{2,1}(a,n).\]	
\end{theorem}
\noindent By Remark~\ref{remark: transcendance}, note that this theorem answers in the affirmative the question of Alon mentioned in Section~\ref{section: introduction}.
 Building on Theorem~\ref{theorem: mubayi Terry conjecture}, we also prove some further general cases of Conjecture~\ref{conjecture: entropy densities}.
\begin{theorem}\label{theorem: (5,5)}
	For all integers $a\geq 2$ and all $n\geq 124$,
	\[\mathrm{ex}_{\Pi}(n, 5, \binom{5}{2}a +5) = \Pi_{2,1}(a,n).\]	
\end{theorem}
\begin{theorem}\label{theorem: (6,7)}
	For all integers $a\geq 2$ and all $n\geq 503$,
	\[\mathrm{ex}_{\Pi}(n, 6, \binom{6}{2}a +7) = \Pi_{2,1}(a,n).\]	
\end{theorem}
\begin{theorem}\label{theorem: (7,9)}
	There exists $N_0\in \mathbb{N}$ such that for all integers $a\geq 2$ and all $n\geq N_0$,
	\[\mathrm{ex}_{\Pi}(n, 7, \binom{7}{2}a +9) = \Pi_{2,1}(a,n).\]	
\end{theorem}
\noindent Applying~\eqref{eq: MT counting}, we have the following counting corollary to Theorems~\ref{theorem: mubayi Terry conjecture}--\ref{theorem: (7,9)}:
\begin{corollary}
For all integers $a\geq 2$ and all $s\in \{4,5,6,7\}$,  
\[\bigl\vert \mathcal{F}(n,s, \Sigma_{2,1}(a,s)-\binom{s}{2})\Bigr\vert = e^{\pi_{2,1}(a)\binom{n}{2}+o(n^2)},\]
with $\pi_{2,1}(a)$ the quantity given in~\eqref{eq: pi(r,d)(a,n)}. 
\end{corollary}

Mubayi and Terry~\cite[Theorem 10]{mt2} also proved that for all $r$ so that $\frac{s}{2}\leq r\leq s-1$ and all $n\geq s$,
 $\mathrm{ex}_{\Pi}(n, s, \Sigma_{r,0}(a,n))= \Pi_{r,0}(a,n)$, which establishes `half' of the $d=0$ case of Conjecture~\ref{conjecture: entropy densities}. In this paper, we complete a proof of the $d=0$ case of Conjecture~\ref{conjecture: entropy densities}, extending the result of Mubayi and Terry, and obtaining the following multigraph Tur\'an theorem:
\begin{theorem}[Multigraph Tur\'an theorem]\label{theorem: turan}
	For all $a,r,s \in \mathbb{N}$ with $s\geq r+1$ and all integers $n\geq 2r(s+2)+ r(s+2)\sqrt{s-1}$,
	\[\mathrm{ex}_{\Pi}\left(n, s, \Sigma_{r,0}(a,s)\right) = \Pi_{r,0}(a,n).\] 
	Furthermore, the set of product-extremal multigraphs is precisely the family $T^P_{r,0}(a,n)$.	
\end{theorem}
As further evidence in favour of Conjecture~\ref{conjecture: entropy densities}, we prove that, for $r,d$ fixed, if Conjecture~\ref{conjecture: entropy densities} holds for pairs $(s,q)=(s,\Sigma_{r,d}(a,s))$ for all $a$ sufficiently large, then  Conjecture~\ref{conjecture: entropy densities} also holds for $(s+1,\Sigma_{r,d}(a,s+1))$ and all $a$ sufficiently large. In other words, one can `step up' and deduce higher cases of Conjecture~\ref{conjecture: entropy densities} from the corresponding lower cases.
\begin{theorem}[Step-up in the conjecture]\label{theorem: step-up}
	Let $r,d\in \mathbb{N}$ and let $s\geq  (r-1)(d+1)+2$. Suppose that there exist $a_0, n_0$ such that for all $a\geq a_0$ and $n\geq n_0$ we have
	\begin{align*}
	\mathrm{ex}_{\Pi}(n, s, \Sigma_{r,d}(a,s)) = \Pi_{r,d}(a,n).
	\end{align*}
	Then there exists $a_1\geq a_0$ and $n_1\geq n_0$ such that for all $a\geq a_1$ and $n\geq n_1$, we have
	\begin{align*}
	\mathrm{ex}_{\Pi}(n, s+1, \Sigma_{r,d}(a,s+1)) = \Pi_{r,d}(a,n).
	\end{align*}
\end{theorem}
In Section~\ref{sec-s} we derive a so-called `Step down in the base case' result (Theorem~\ref{theorem: step down}) which may prove useful in resolving Conjecture~\ref{conjecture: entropy densities} and also provides more evidence in support of it.

\smallskip

Whilst we conjecture that Construction~\ref{construction: lower bound} yields extremal examples for the problem of determining $\mathrm{ex}_{\Pi}( s, q)$ for many choices of $q$, in general these constructions will not be extremal for other choices. In Section~\ref{construct} we provide other more intricate constructions based on iterated versions of Construction~\ref{construction: lower bound} that demonstrate this. 

We also prove the following theorem, which determines the asymptotic behaviour of $\mathrm{ex}_{\Pi}(n, s,q)$ in the `sparse' case where $q\leq 2\binom{s}{2}$. 
\begin{theorem}[Sparse case]\label{theorem: sparse case}
	Let $s\in \mathbb{Z}_{\geq 2}$ and let $q$ be an integer with $0\leq q \leq 2\binom{s}{2}$. Then the following holds:
	\begin{align*}
	\mathrm{ex}_{\Pi}(n,s,q)=\left\{ \begin{array}{ll}
	0 & \textrm{if } 0\leq q< \binom{s}{2}\textrm{ and } n\geq s\\
	1 & \textrm{if }q=\binom{s}{2}\textrm{ and } n\geq s\\
	2^{q-\binom{s}{2}} & \textrm{if }\binom{s}{2} < q < \binom{s}{2}+ \lfloor \frac{s}{2}\rfloor \textrm{ and }n\geq s\\
	2^{\Theta(n)} & \textrm{if }\binom{s}{2}+ \lfloor \frac{s}{2}\rfloor \leq q < \binom{s}{2}+ s-2\\
	2^{\bigl\lfloor \left(\frac{s-2}{s-1}\right)n\bigr\rfloor} & \textrm{if }q = \binom{s}{2}+ s-2 \textrm{ and } n\geq s\\
	2^{\mathrm{ex}(n, \{C_3, C_4, \ldots, C_{s}\})} & \textrm{if }q = \binom{s}{2}+ s-1 \textrm{ and }n\geq s\\
	2^{o(n^2)} & \textrm{if } \binom{s}{2}+ s\leq q< \binom{s}{2}+ \lfloor\frac{s^2}{4}\rfloor \\
	2^{\Theta(n^2)} & \textrm{if } \binom{s}{2}+\lfloor\frac{s^2}{4}\rfloor \leq q \leq 2\binom{s}{2}.
	\end{array}\right.
	\end{align*}
\end{theorem}
\noindent Note that in fact Theorem~\ref{theorem: turan} gives more precise information in the range $\binom{s}{2}+\lfloor\frac{s^2}{4}\rfloor \leq q \leq 2\binom{s}{2}$, showing that for all $n$ sufficiently large,
\[\mathrm{ex}_{\Pi}(n,s,q)= 2^{\mathrm{ex}(n, K_r)} \qquad \textrm{if }	q=\binom{s}{2}+ \mathrm{ex}(s, K_r) \textrm{ for some }r \textrm{ with } \ 3\leq r \leq s-1.\]
\noindent (Here $\mathrm{ex}(n, K_r)$ denotes as usual the Tur\'an number of the complete graph $K_r$ on $r$ vertices.) Theorem~\ref{theorem: sparse case} generalises and extends earlier results of Mubayi and Terry, who proved the case $q=\binom{s}{2}+s-2$ and the case $s=4$, $q=\binom{s}{2}+s-1=9$~\cite[Theorems 8(b) and 11]{mt2}.

 By a standard first moment and alterations argument, it is easy to see that $\mathrm{ex}(n, \{C_3, C_4, \ldots, C_{s}\})$ is of order at least $n^{1+\frac{1}{s-1}}$. Thus  Theorem~\ref{theorem: sparse case} shows that $\log \Bigl(\mathrm{ex}_{\Pi}(n,s,q)\Bigr)$ is of constant order when $\binom{s}{2}\leq q <\binom{s}{2}+\lfloor \frac{s}{2}\rfloor$, linear in $n$ when $\binom{s}{2}+\lfloor \frac{s}{2}\rfloor\leq q< \binom{s}{2}+s-1$, superlinear but subquadratic when $ \binom{s}{2}+s-1\leq q <  \binom{s}{2}+\lfloor\frac{s^2}{4}\rfloor $ and quadratic thereafter.

\smallskip 

{\bf Remark.} Since submitting the paper, the second author~\cite{vfr} has determined the value of $\text{ex}_\Pi \left (2r, \Sigma_{r,1}(a,2r) \right)$ for all integers $a,r \geq 2$. This result, combined with the proof of Theorem~\ref{theorem: step-up} implies that Conjecture~\ref{conjecture: entropy densities} is asymptotically true for $d=1$ and $a$ sufficiently large.

\section{Overview of our methodology}
Many of our proofs employ the same  underlying approach, which we now outline.
Given $s,q \in \mathbb N$, suppose one wishes to upper-bound $\text{ex}_\Pi (n,s,q)$ for all large enough $n\in \mathbb N$ by some parameter $ P_{n,s,q}$. Then one must show that given any $G \in \mathcal F(n,s,q)$, $P(G)\leq P_{n,s,q}$.

Our first step will typically be to show that given any such $G \in \mathcal F(n,s,q)$,  either
\begin{itemize}
\item[(i)] $G$ has a vertex $v$ whose product-degree is significantly smaller than the average product-degree of our candidate extremal multigraph in 
$\mathcal F(n,s,q)$, or \item[(ii)] $G \in  \mathcal F(n,s-1,q')$ for some well-chosen choice of $q' \in \mathbb N$.
\end{itemize}
Such assertions will be proven using the integral version of the AM-GM inequality
(Proposition~\ref{prop: integral AM-GM} below).

One can think of (ii) as `stepping down', allowing us to use previously obtained information about $\mathcal F(n,s-1,q')$. In particular, if we already know that 
$\text{ex}_\Pi (n,s-1,q')\leq P_{n,s,q}$ then in Case~(ii) we immediately obtain that
$P(G) \leq P_{n,s,q}$, as desired.

In Case (i), we  consider $G_{n-1}:=G\setminus \{v\}$. Repeating the previous argument we will again conclude $G_{n-1}$ has a vertex of relatively small product-degree or $G_{n-1} \in \mathcal F(n-1,s-1,q')$.

Thus, eventually we either produce many small product-degree vertices in $G$ or obtain some structural information about a subgraph of $G$; specifically that some $i$-vertex induced subgraph $G_i$ of $G$ belongs to $\mathcal F(i,s-1,q')$. In the former case this will  (combined with some information on the value of 
$\text{ex}_\Pi (m,s,q)$ for a small  $m\in \mathbb N$)  be enough to 
conclude that $P(G) \leq P_{n,s,q}$. In the latter case we may need to seek further information about $G_i$. For example, similar arguments may allow us to conclude $G_i$ has many vertices of low product-degree or contains an $i'$-vertex induced subgraph $G_{i'}$ that belongs to
$\mathcal F(i',s-2,q'')$ for some carefully chosen $q'' \in \mathbb N$.

Repeating such arguments will then give us enough structural information about $G$ to conclude $P(G) \leq P_{n,s,q}$.
For example, in the proof of Theorem~\ref{theorem: mubayi Terry conjecture} our argument allows us to restrict  to the case when every edge in $G$ has multiplicity $a-1,a$ or $a+1$ and every triple of vertices spans at most $3a+2$ edges. Then to upper-bound $P(G)$ we use a `vertex cloning' procedure to obtain a very well-structured intermediate multigraph $G'$ where $P(G)\leq P(G')$ (see Proposition~\ref{prop: modifying G to get the clique structure}). As $G'$ has  much cleaner structure compared to $G$, it is not too difficult to appropriately upper-bound $P(G')$ (and therefore $P(G)$).

\section{Preliminary results}\label{sec4}

In this section we prove a number of preliminary results that will be useful at various points in this paper.
\subsection{Properties of  $T^e_{r,d}(a,s)$ and $x_{r\star}(a,d)$}

The following proposition, together with the discussion after the statement of Conjecture~\ref{conjecture: entropy densities}, explains why the latter conjecture
involves the condition that  $s\geq (r-1)(d+1)+2$.
\begin{proposition}\label{prop: sum-extremal subgraphs, threshold for 2 vertices in extremal part}
Let $r\in \mathbb{Z}_{\geq 2}$, $a\in \mathbb{N}$, $d\in [0,a-1]$. Given $G\in T^e_{r,d}(a,s)$ we let $\sqcup_{i=0}^{r-1}V_i$ be the canonical partition of $G$. Then the following hold:
\begin{enumerate}[(i)]
	\item if $s\geq (r-1)(d+1)+2$, then there is a  $G\in T^e_{r,d}(a,s)$  so that $\vert V_0 \vert \geq 2$;
	\item if $s\leq (r-1)(d+2)+1$,  then there is a $G\in T^e_{r,d}(a,s)$  so that $\vert V_0\vert \leq 1$.
\end{enumerate}
\end{proposition}
\noindent Proposition~\ref{prop: sum-extremal subgraphs, threshold for 2 vertices in extremal part} immediately implies that $s=(r-1)(d+1)+2$ is the threshold at which $\Sigma_{r,d}(a,s)<\Sigma_{r,d-1}(a,s)$ begins to hold.
\begin{corollary}\label{corollary: why (r-1)(d+1)+2}
For all integers $a,r,d$ with $a\geq 2$, $r\geq 2$ and $0\leq d <a$, we have
\begin{align*}
\Sigma_{r, d-1}(a,s)>\Sigma_{r, d}(a,s) && \textrm{if }&s\geq (r-1)(d+1)+2,\\
\Sigma_{r, d-1}(a,s)=\Sigma_{r, d}(a,s) && \textrm{if }&s<  (r-1)(d+1)+2.
\end{align*}
\end{corollary}
\begin{proof}
Clearly we always have $\Sigma_{r, d-1}(a,s)\geq \Sigma_{r, d}(a,s)$. For $s\geq (r-1)(d+1)+2$, Proposition~\ref{prop: sum-extremal subgraphs, threshold for 2 vertices in extremal part}(i) implies there exists a multigraph $G\in T^e_{r,d}(a,s)$ so that $\vert V_0\vert \geq 2$. Considering the multigraph $G'$ in $\mathcal{T}_{r,d-1}(a,s)$ having the same canonical partition as $G$ then shows
\begin{align*}
\Sigma_{r,d-1}(a,s)\geq e(G')> e(G)=\Sigma_{r,d}(a,s).
\end{align*}
For $s< (r+1)(d+1)+2$, Proposition~\ref{prop: sum-extremal subgraphs, threshold for 2 vertices in extremal part}(ii) implies there exists a multigraph $G\in T^e_{r,d-1}(a,s)$ so that $\vert V_0\vert \leq 1$. Such a multigraph $G$ can be viewed as a member of $\mathcal{T}_{r,d}(a,s)$, whence $\Sigma_{r, d-1}(a,s)\leq \Sigma_{r, d}(a,s)$ and thus $\Sigma_{r, d-1}(a,s)=\Sigma_{r, d}(a,s)$. 	
\end{proof}
\begin{proof}[Proof of Proposition~\ref{prop: sum-extremal subgraphs, threshold for 2 vertices in extremal part}]
Let $G\in T^e_{r,d}(a,s)$, and let 	$\sqcup_{i=0}^{r-1}V_i$  be the associated partition of $V(G)$. Set $x:=\vert V_0\vert $. By Tur\'an's theorem and the maximality of $e(G)$, we have that $\sqcup_{i=1}^{r-1}V_i$ forms an equipartition of $[s]\setminus V_0$, and we may assume in particular that $\vert V_{r-1}\vert = \Bigl\lceil \frac{n-x}{r-1}\Bigr\rceil$ and $\vert V_1\vert = \Bigl\lfloor \frac{n-x}{r-1}\Bigr\rfloor$.

Suppose $s\geq(r-1)(d+1)+2$, and $x\leq 1$. Then $\vert V_{r-1}\vert \geq d+2$, and moving any vertex from $V_{r-1}$ to $V_0$ increases $e(G)$ by $\left(\vert V_{r-1}\vert -1\right)- (d+1)x\geq 0$. Thus we can achieve $x>1$ while remaining inside $T^e_{r,d}(a,s)$. This establishes (i).

Similarly, suppose $s\leq(r-1)(d+2)+1$ and $x\geq 2$. Then $\vert V_1\vert \leq d+1$, and moving any vertex from $V_0$ to $V_1$ increases $e(G)$ by $(d+1)(x-1)-\vert V_1\vert\geq 0$. Thus we can achieve $x<2$ while remaining inside $T^e_{r,d}(a,s)$. This establishes (ii).
\end{proof}
\noindent Using similar arguments to the ones above, one can in fact establish precisely the possible values for $\vert V_0\vert$.

\begin{proposition}[Growth and partition sizes of $T^e_{r,d}(a,n)$]\label{prop: sum-extremal}
	Let $r,a\in \mathbb{N}, d\in [0,a-1]$. Let $n\geq rd+r-d$, and let $t$ be the remainder when $n$ is divided by $rd+r-d$. Then 
	\[\Sigma_{r,d}(a,n+1)-\Sigma_{r,d}(a,n)=\left\{\begin{array}{ll}
(a+1)n-(d+1)\Bigl\lfloor\frac{n}{rd+r-d}\Bigr\rfloor & \textrm{ if }0\leq t\leq r-1,\\[10pt]
(a+1)n - (d+1) \Bigl\lfloor \frac{n}{rd+r-d}\Bigr \rfloor -  \Bigl \lfloor \frac{t-1}{r-1}\Bigr \rfloor
& \textrm{ if }r\leq t\leq rd+r-d-1.
\end{array} \right .\]
Furthermore, if $G\in T^e_{r,d}(a,n)$ has canonical partition $\sqcup_{i=0}^{r-1}V_i$, then 
\begin{align*}
\vert V_0\vert =\left\{\begin{array}{cl}
\frac{n}{rd+r-d} & \textrm{if } t=0,\\[10pt]
	\Bigl\lfloor\frac{n}{rd+r-d}\Bigr\rfloor  \textrm{ or } \Bigl\lceil\frac{n}{rd+r-d}\Bigr\rceil  & \textrm{if } 1\leq t \leq r-1,\\[10pt]
	\Bigl\lceil\frac{n}{rd+r-d}\Bigr\rceil  & \textrm{if } r\leq t\leq rd+r-d-1.
\end{array}\right.
\end{align*}
\end{proposition}
\begin{proof}
Let $G\in T^e_{r,d}(a,n)$.	Let $\sqcup_{i=0}^{r-1}V_i$ be the associated partition of $V(G)=[n]$, and let  $v_i:=\vert V_i\vert$. The main work in the proposition is determining the possible values of the $v_i$s. We may assume without loss of generality that 
\[v_0\leq v_1 \leq v_2\leq \ldots \leq v_{r-1}.\]
By extremality of $T^e_{r,d}(a,n)$, we have that, for $i\in [r-1]$, $v_i=\Bigr\lfloor \frac{n-v_0+i-1}{r-1}\Bigr\rfloor$. Since $n\geq r$, we have that $v_0\geq 1$. By extremality, we know that the number of edges cannot increase if we shift a vertex from $V_0$ to $V_1$ or if we shift a vertex from $V_{r-1}$ to $V_0$. This implies that
\begin{align*}
\left(v_0-1\right)(d+1) - v_1\leq 0 \ \ \text{ and } \ \ -(d+1)v_0+ \left(v_{r-1}-1\right)\leq 0.
\end{align*}	
Using $v_1\leq \frac{n-v_0}{r-1}\leq v_{r-1}$ and rearranging these two inequalities, we get
\begin{align*}
\frac{1}{rd+r-d}\left(n-(r-1)\right)\leq v_0 \leq \frac{1}{rd+r-d}\left(n+ (d+1)(r-1)\right).
\end{align*}
Let $q$ and $t$ denote the quotient and remainder of $n$ when divided by $
rd+r-d$. The above inequality may thus be rewritten as
\[q-\frac{(r-1)-t}{rd+r-d}\leq v_0 \leq q+1 +\frac{t-1}{rd+r-d}. \]
Since $q$ is an integer and $t<rd+r-d$, the right-hand side inequality implies that $v_0\leq q+1$, and that $v_0\leq q$ if $t=0$. The left-hand side inequality, for its part, implies that $v_0\geq q$ and that for $t>r-1$ we have $v_0\geq q+1$. Summarising we have
\begin{align*}\label{eq: bounds on v_0}
v_0=\left\{\begin{array}{ll}
q & \textrm{if } t=0, \\
q \textrm{ or } q+1 & \textrm{if } 1\leq t \leq r-1, \\
q+1 & \textrm{if } r\leq t\leq rd+r-d-1.
\end{array} \right.\end{align*}
It readily follows from this that
\[v_1=\left\{\begin{array}{ll}
q(d+1) & \textrm{if } 0\leq t\leq r-2, \\
q(d+1)\textrm{ or } q(d+1)+1 & \textrm{if } t=r-1, \\
q(d+1) +\lfloor \frac{t-1}{r-1}\rfloor & \textrm{if } r\leq t\leq rd+r-d-1.
\end{array} \right.\]
Moreover at $t=r-1$, we have that $v_1=q(d+1)$ if and only if $v_0=q+1$. From this, we may easily read off the value of $\Sigma_{r,d}(a, n+1)-\Sigma_{r,d}(a,n)$: to obtain a multigraph in $T^{e}_{r,d}(a,n+1)$ from $G$, one must add a vertex to either $V_0$ or $V_1$. The former increases the number of edges by $(a+1)n-(d+1)v_0$ and the latter by $(a+1)n-v_1$. Thus
\begin{align*}
	\Sigma_{r,d}(a, n+1)-\Sigma_{r,d}(a,n)&=(a+1)n-\min\left((d+1)v_0, v_1\right)\\
&=\left\{ \begin{array}{ll}
(a+1)n - (d+1)q & \textrm{if }0\leq t\leq r-1\\
(a+1)n - q(d+1)-  \lfloor \frac{t-1}{r-1}\rfloor & \textrm{if }r\leq t\leq rd+r-d-1.
\end{array} \right.\end{align*}
Substituting $\Bigl\lfloor n/(rd+r-d)	\Bigr\rfloor$ for $q$ yields the statement of the proposition.
\end{proof}

\begin{proposition}\label{prop: a-monotonicity of xstar}
For any $r\in \mathbb{N}$, the function $x_{r\star}=x_{r\star}(a,d)$ is monotone increasing in $a$ (over the interval $a\in[d+1, \infty)$) and monotone decreasing in $d$ (over the interval $d\in [0, a-1]$). In particular, for any pair of integers $a>d>1$ in its domain, $x_{r\star}$ satisfies
\[  \frac{\log (d+2) -\log (d+1)}{r \log(d+2) -\log (d+1)}=x_{r\star}(d+1,d) \leq x_{r\star}(a,d) < \lim_{a \rightarrow \infty}x_{r\star}(a,d)= \frac{1}{d(r-1)+r}. \]
\end{proposition}
\begin{proof}
This is a simple exercise in calculus.
\end{proof}

\begin{proposition}\label{prop:diff in sigma when adding one vertex}
	For all natural numbers $r,d$ and $s\geq (r-1)(d+1)+2$, there exists $a_0$ such that for all $a\geq a_0$,
	\begin{align*}
\Sigma_{r,d}(a,s+1)-\Sigma_{r,d}(a,s)-1< \left(a +\frac{r-2+x_{r\star}(a,d)}{r-1}\right)s.
	\end{align*}
\end{proposition}
\begin{proof}
By Proposition~\ref{prop: sum-extremal}, we have that 
\begin{align*}
\Sigma_{r,d}(a,s+1)-\Sigma_{r,d}(a,s)-1&< (a+1)s- (d+1)\frac{s}{rd+r-d}=\left(a+\frac{r-2}{r-1} +\frac{1}{(r-1)(rd+r-d)}\right)s.
\end{align*}
	By Proposition~\ref{prop: a-monotonicity of xstar}, for any fixed $r\in \mathbb{N}$, we have that $x_{r \star}=x_{r \star}(a,d)$ is monotonically increasing in $a$, and tends to $1/(rd+r-d)$ as $a\rightarrow \infty$. In particular, for all $a$ sufficiently large and any $c<\frac{r-2}{r-1}+\frac{(1/(rd+r-d))}{r-1}$, we have $c< \frac{r-2+x_{r \star}(a,d)}{r-1}$. Taken together with the inequality above, this establishes the proposition.	
\end{proof}

\subsection{The AM--GM inequality}
An important tool we will make use of is the integral version of the AM--GM inequality:
\begin{proposition}\label{prop: integral AM-GM}
Let $a,n\in \mathbb{N}$, $t\in  [0,n]$, and let $w_1, \ldots, w_n$ be non-negative integers with $\sum_{i=1}^n w_i= an +t$. Then the following hold:
\begin{enumerate}[(i)]
	\item $\prod_{i=1}^n w_i \leq a^{n-t}(a+1)^t$; 
	\item if $t\leq n-2$ and $w_1=a-1$, then $\prod_{i=1}^n w_i\leq (a-1)a^{n-t-2}(a+1)^{t+1}$.
\end{enumerate}
\end{proposition}	
\begin{proof}
Part (i) is a standard integral version of the AM--GM inequality, see e.g. Lemma~15 in~\cite{mt2}. Part (ii) is an immediate corollary of part (i). 
\end{proof}

\section{Stepping down in the base case}\label{sec-s}
Recall that Theorem~\ref{theorem: step-up} allows us to deduce higher cases of Conjecture~\ref{conjecture: entropy densities} from the corresponding lower cases.
The following result may be useful for proving the `base cases' ($s=(r-1)(d+1)+2$) of the conjecture.
The proof also  gives an example of the type of ideas that we will be using in this paper, and in particular, how we make use  of the integral AM-GM inequality.

\begin{theorem}[Step down in the base case]\label{theorem: step down}
For all $r,d\in \mathbb N$ there exist $\varepsilon>0$ and $n_0\in \mathbb N$ such that for all $a\geq d+1$ and all $n\geq n_0$, if $G\in \mathcal{F}(n+1, s+1, \Sigma_{r,d}(a,s+1))$ where $ s \in \mathbb N$ such that $2 \leq s \leq  (r-1)(d+1)+1$, then either $G\in \mathcal{F}(n+1, s, \Sigma_{r,d}(a,s))$ or $G$ contains a vertex $u_0$ with product-degree
\begin{align}\label{goodbound}
p_G(u_0) \leq \frac{\Pi_{r,d}(a,n+1)}{\Pi_{r,d}(a,n)}e^{-\varepsilon n}.
\end{align}
\end{theorem}
\begin{proof}
Let $\varepsilon >0$ be sufficiently small and $n$ sufficiently large.
Let $G$ be a graph in $\mathcal{F}(n+1,s+1,\Sigma_{r,d}(a,s+1))$, and write $V:=V(G)$. 
 Suppose $G$ contains an $s$-set  $U$ such that $e(G[U])\geq \Sigma_{r,d}(a, s)+1$. Our assumption on $G$ tells us that (i) $e(G[U])\leq \Sigma_{r,d}(a,s+1)$, and (ii) every $v\in V\setminus U$ sends at most $\Sigma_{r,d}(a,s+1)-\Sigma_{r,d}(a,s)-1$ edges into $U$. We shall deduce from this that $U$ contains a low product-degree vertex. To do this we consider two cases.

\noindent\textbf{Case 1: $2\leq s\leq r-1$}. For such values of $s$, by Proposition~\ref{prop: sum-extremal} we have $\Sigma_{r,d}(a,s+1)-\Sigma_{r,d}(a,s)-1=as+(s-1)$.  The integral AM--GM inequality (Proposition~\ref{prop: integral AM-GM}(i)) tells us that $\prod_{u \in U}w(uv) \leqslant a(a+1)^{s-1} $ for each vertex $v \in V \setminus U$.  Thus, we have that
\begin{align*}
\prod_{u\in U}p_{G}(u) &= P(G[U])^2 \prod_{u\in U, v \in V \setminus U}w(uv) \\
&\leq P(G[U])^2 a^{n-s+1}(a+1)^{(s-1)(n-s+1)}.
\end{align*} 
Note that (i) above implies  $P(G[U])^2$ is bounded above by a constant term (i.e., independent of $n$). Thus,
averaging over $u\in U$, the inequality above implies the existence of a vertex $u_0\in U$ with 
\begin{align}\label{quseful}
p_G(u_0)\leq    a^{n}\left (\frac{a+1}{a} \right )^{\frac{(s-1)n}{s}+O(1)}.
\end{align}

Recall that from the definition of $T^P_{r,d}(a,n)$ and $x_{r \star}$, all vertices in graphs from $T^P_{r,d}(a,n)$ have product-degree  $\left(\frac{a+1}{a}\right)^{\frac{r-2+x_{r \star}}{r-1}n+O(1)}a^{n}$.
Thus, 
\begin{align}\label{veryuseful}
\frac{\Pi_{r,d}(a,n+1)}{\Pi_{r,d}(a,n)} =
\left(\frac{a+1}{a}\right)^{\frac{r-2+x_{r \star}}{r-1}n+O(1)}a^{n}.
\end{align}

Since $\frac{s-1}{s}\leq \frac{r-2}{r-1}$, comparing (\ref{quseful}) with (\ref{veryuseful}) (and since $\varepsilon>0$ is sufficiently small and $n$ is sufficiently large), we see that $u_0$ satisfies (\ref{goodbound}), as desired.

\noindent\textbf{Case 2: $r\leq s\leq (r-1)(d+1)+1$}. Write $s =(r-1)q+\ell$, where $q, \ell$ are respectively the quotient and remainder of $s$  upon division by $r-1$. Then, using Proposition~\ref{prop: sum-extremal subgraphs, threshold for 2 vertices in extremal part}(ii) (which tells us that in optimal canonical partitions for $T^e_{r,d}(a,s)$ and $T^e_{r,d}(a,s+1)$ we may take $\vert V_0\vert =1$),
 we see that 
\begin{align*}
\Sigma_{r,d}(a, s+1)-\Sigma_{r,d}(a,s)-1&=\left\{\begin{array}{ll}
as+s-q & \textrm{if } \ell=0,\\
as+s-q-1& \textrm{if } \ell>0,\end{array} \right.\\
&= as+ s-\left \lceil \frac{s}{r-1} \right\rceil.\end{align*}
Applying the integral AM--GM-inequality (Proposition~\ref{prop: integral AM-GM}(i)), we have 
\begin{align*}
\prod_{u\in U}p_{G}(u)\leq P(G[U])^2 a^{\lceil\frac{s}{r-1}\rceil n}(a+1)^{(s-\lceil\frac{s}{r-1}\rceil)n}\leq a^{sn}\left(\frac{a+1}{a}\right)^{sn\left(\frac{r-2}{r-1}\right)+O(1)},
\end{align*} 
where as in Case 1 we use that $P(G[U])^2$ is bounded above by a constant term (due to (i)).
Averaging over $u\in U$, we see this implies the existence of a vertex $u_0\in U$ with
\[p_G(u_0)\leq   a^{n}\left(\frac{a+1}{a}\right)^{n\left(\frac{r-2}{r-1}\right)+O(1)}.\]
Then similarly to Case~1, comparing this with (\ref{veryuseful}), we can conclude $u_0$ satisfies (\ref{goodbound}).

\end{proof}

\section{Proof of Theorem~\ref{theorem: step-up}}
We now use  the tools from Section~\ref{sec4} to prove
Theorem~\ref{theorem: step-up}.

\begin{proof}[Proof of Theorem~\ref{theorem: step-up}]
By Proposition~\ref{prop:diff in sigma when adding one vertex}, we may fix some $a_1\geq a_0$ such that, for all $a \geq a_1$, we have 
\[\varepsilon := \left(a +\frac{r-2+x_{r \star}(a,d)}{r-1}\right)s-\Bigl(\Sigma_{r,d}(a,s+1)-\Sigma_{r,d}(a,s)-1\Bigr)>0.\]

Let $N\geq n_0$ be a sufficiently large constant to be determined later, and consider a multigraph $G\in \mathcal{F}\left(n, s+1, \Sigma_{r,d}(a,s+1)\right)$ for some $n\geq n_1:=N\left(1+ \frac{s}{\varepsilon}\right)$.  We sequentially remove vertices with minimum product-degree from $G$ to obtain a sequence $G=G_n, G_{n-1},G_{n-2}, \ldots,$ of multigraphs, where each $G_i$ is an induced subgraph of $G$
on $i$ vertices, stopping when we have obtained a multigraph $G_{n'}$ such that either $n'=N$ or $G_{n'} \in \mathcal{F}(n', s, \Sigma_{r,d}(a,s))$.

Suppose that there is some $s$-set $U$ in $V(G_i)$ spanning at least $\Sigma_{r,d}(a,s)+1+x$ edges for some $x\geq 0$. Then every vertex $v\in V(G_i)\setminus U$ sends at most $\Sigma_{r,d}(a,s+1)-\Sigma_{r,d}(a,s)-1-x$ edges into $U$.  As $\Sigma_{r,d}(a,s)+1+x \leq {s \choose 2}(a+1)+x$,
 the standard AM-GM inequality tells us that 
\begin{equation}
P(G_i[U]) \leqslant \left( a+1+\frac{2x}{s(s-1)} \right)^{{s \choose 2}}. \nonumber
\end{equation}
Note that 
$as +\frac{r-2 + x_{r \star}}{r-1}s - (\varepsilon + x)=\Sigma_{r,d}(a,s+1)-\Sigma_{r,d}(a,s)-1-x$.
Thus, Proposition~\ref{prop: integral AM-GM}(i) (with $s$ and $\frac{r-2 + x_{r \star}}{r-1}s - (\varepsilon + x)$ playing the roles of $n$ and $t$ respectively) implies that
 each vertex $v \in V(G_i) \setminus U$ has edge product-degree into $U$ of at most
 \begin{equation}
\prod_{u\in U} w_{G_i}(uv)\leq a^{s}\left( \frac{a+1}{a}\right)^{ \frac{r-2 + x_{r \star}}{r-1}s - (\varepsilon + x)}. \nonumber
\end{equation}
Thus, using a similar argument to that found in the proof of Theorem~\ref{theorem: step down}, we have that there is some vertex in $U$ with product-degree in $G_i$ of at most 
\begin{align}\label{prelimsec, eq: bound on product degree of a vertex from bad s-set}
\left( a+1+\frac{2x}{s(s-1)}\right)^{s-1} a^{i-s}  \left(\frac{a+1}{a}\right)^{(i-s)(\frac{r-2+x_{r \star}}{r-1} -\frac{(\varepsilon+x)}{s})}.\end{align}
Now, we may assume $x\leq (a+1)s$, since otherwise adding any vertex of $G_i$ to $U$ would give rise to an $(s+1)$-set spanning strictly more than $\Sigma_{r,d}(a, s+1)$ edges.  Recall that from the definition of $T^P_{r,d}(a,i)$ and $x_{r \star}$, all vertices in graphs from $T^P_{r,d}(a,i)$ have product-degree  $\left(\frac{a+1}{a}\right)^{\frac{r-2+x_{r \star}}{r-1}i+O(1)}a^{i}$. Combining this with \eqref{prelimsec, eq: bound on product degree of a vertex from bad s-set}, we have that for all $i\geq N$  and $N$ chosen sufficiently large, $G_i$ contains a vertex with  product-degree at most
\[\frac{\Pi_{r,d}(a,i)}{\Pi_{r,d}(a,i-1)}\left(\frac{a+1}{a}\right)^{-\frac{\varepsilon}{2s}i} \leq \frac{\Pi_{r,d}(a,i)}{\Pi_{r,d}(a,i-1)}\left(\frac{a+1}{a}\right)^{-\frac{\varepsilon}{2s}N}  .\]

It follows that for every $i$ with $n'\leq i < n$, $G_i$ satisfies
\begin{align}\label{prelimsec eq: bound on P(G)}
P(G)=P(G_n)\leq P(G_i)\prod_{k=i}^{n-1} \frac{\Pi_{r,d}(a,k+1)}{\Pi_{r,d}(a,k)}\left(\frac{a+1}{a}\right)^{-\frac{\varepsilon}{2s}N}\leq P(G_i) \frac{\Pi_{r,d}(a,n)}{\Pi_{r,d}(a,i)}\left(\frac{a+1}{a}\right)^{-\frac{\varepsilon N}{2s}(n-i)}.
\end{align}
Consider now the case $i=n'$. If $n'=N$, then note that by averaging over all $(s+1)$-sets, we have that the average edge-multiplicity in a multigraph from $\mathcal{F}(N, s+1, \Sigma_{r,d}(a,s+1) )$ is at most $a+1$. It thus follows from the AM-GM inequality that $P(G_{N})\leq \mathrm{ex}_{\Pi}(N, s+1, \Sigma_{r,d}(a, s+1))\leq (a+1)^{\binom{N}{2}}< \left(\frac{a+1}{a}\right)^{\binom{N}{2}}\Pi_{r,d}(a, N)$. Combining this with \eqref{prelimsec eq: bound on P(G)} and our assumption $n-N\geq n_1-N= N\cdot \frac{s}{\varepsilon}$, we obtain
\begin{eqnarray*}
P(G) &\leq&  P(G_N)\frac{\Pi_{r,d}(a,n)}{\Pi_{r,d}(a,N)}\left(\frac{a+1}{a}\right)^{-\frac{\varepsilon N}{2s}(n-N)} \\
&\leq& \left(\frac{a+1}{a}\right)^{\binom{N}{2}} \Pi_{r,d}(a, n)\left(\frac{a+1}{a}\right)^{-\frac{N^{2}}{2}} \\
&<& \Pi_{r,d}(a, n), \nonumber
\end{eqnarray*}
as desired.

On the other hand if $n'>N$  then $G_{n'}\in \mathcal{F}\left(n', s, \Sigma_{r,d}(a,s)\right)$. So by our assumption (and the fact that $N\geq n_0$) we have $P(G_{n'})\leq \Pi_{r,d}(a,n')$. Combining this with \eqref{prelimsec eq: bound on P(G)}, we get 
\[P(G)\leq P(G_{n'}) \frac{\Pi_{r,d}(a,n)}{\Pi_{r,d}(a,n')}\leq \Pi_{r,d}(a, n),\]
as desired.
\end{proof}

\section{Proof of the Multigraph Tur\'{a}n theorem}

In this section, we prove Theorem~\ref{theorem: turan}. The crux of the argument is the following lemma.
\begin{lemma}\label{lemma: turan degree-removal}
Let $r,s,a,n \in \mathbb N$ be such that $n\geq r(s+2)$ and $r,s\geq 2$. Let  $G$ be a multigraph from  $\mathcal{F}\left(n+1, s+1,\Sigma_{r,0}(a,s+1) \right)$. Then either $G\in \mathcal{F}\left(n+1, s, \Sigma_{r,0}(a,s)\right)$ or $G$ contains a vertex of product-degree strictly less than 
\[\left(\frac{a}{a+1}\right)^{\frac{n-r(s+2)}{rs}}\Pi_{r,0}(a,n+1)/\Pi_{r,0}(a,n).\]
\end{lemma}
\begin{proof}
Suppose $G$ contains an $s$-set $U$ such that $e(G[U])= \Sigma_{r,0}(a,s)+1+x$, for some  $x\geq 0$. By our assumption that no $(s+1)$-set in $G$ spans more than $\Sigma_{r,0}(a, s+1)$ edges (counting multiplicities), we have that every vertex in $V(G)\setminus U$ sends at most 
\begin{align}
\Sigma_{r,0}(a,s+1)-\Sigma_{r,0}(a,s)-1-x= (a+1)s -\Bigl\lfloor \frac{s}{r}\Bigr \rfloor -1 -x\notag
\end{align} 
edges into $U$ (counting multiplicities). 
Summing over the degrees of the vertices in $U$, we thus have
\begin{align}\label{eq: turan case bound on sum of degrees inside X}
\sum_{u\in U} d(u)& \leq 2\left(\Sigma_{r,0}(a,s) +1+x\right) +(n+1-s)\left((a+1)s -\Bigl\lfloor \frac{s}{r}\Bigr \rfloor -1 -x\right)\\ \notag
&\leq s\left((a+1)(s-1)- \left( \frac{s}{r} -1\right)\right) - (1+x)(n-s-1)+(n+1-s)\left((a+1)s-\Bigl\lfloor \frac{s}{r}\Bigr\rfloor\right)\\ \notag
&= (a+1)sn - n \Bigl\lfloor \frac{s}{r}\Bigr\rfloor -(1+x)\left(n-s-1\right)+(s-1)\Bigl\lfloor \frac{s}{r}\Bigr\rfloor -\frac{s^2}{r}+s.
\end{align}	
On the other hand, we have
\begin{align}\label{eq: turan case, bound on edge increase}
s\Bigl(\Sigma_{r,0}(a,n+1) -\Sigma_{r,0}(a,n)\Bigr) &= s\left((a+1)n - \Bigl\lfloor \frac{n}{r}\Big\rfloor \right)= (a+1)sn - s \Bigl\lfloor \frac{n}{r}\Big\rfloor .
\end{align}
Combining ~\eqref{eq: turan case bound on sum of degrees inside X} and ~\eqref{eq: turan case, bound on edge increase}, we have that 
\begin{align}\label{eq: turan case, lower bound on diff degree sum}
& s\left(\Sigma_{r,0}(a,n+1) -\Sigma_{r,0}(a,n)\right) - \sum_{u\in U} d(u)\\\notag
&\qquad \geq n \Bigl\lfloor \frac{s}{r}\Big\rfloor- s\Bigl\lfloor \frac{n}{r}\Big\rfloor + (1+x)(n-s-1)-\left(s +(s-1)\Bigl\lfloor \frac{s}{r}\Bigr\rfloor -\frac{s^2}{r} \right)=: f(n,r,s, x).
\end{align}
We claim that $f(n,r,s, x)>0$. Indeed, observe first of all that since $n\geq r(s+2)$, $f(n,r,s,x)$ is an increasing function of $x$, and thus  $f(n,r,s,x)\geq f(n,r,s,0)$. Let $t$ be the remainder when $s$ is divided by $r$.  Then 
\begin{align*}
f(n,r,s,x)\geq f(n,r,s,0)&\geq  n \left(\frac{s-t}{r}\right)-\frac{sn}{r}+ n-2s-1+ \frac{s+t(s-1)}{r}\\
&=n-2s-1+\frac{s}{r}-\frac{t\left(n-s+1\right)}{r}\\
&\geq n-2s -1+ \frac{s-(r-1)\left(n-s+1\right)}{{r}}>\frac{n}{r}-(s+2)\geq 0,
\end{align*}
using our assumption that $n\geq r(s+2)$. In particular, by averaging, \eqref{eq: turan case, lower bound on diff degree sum} implies some vertex $u\in U$ must have degree $d(u)$ satisfying
\begin{align}\label{eq: bound on deficiency turan}
d(u)+ \frac{n-r(s+2)}{rs} < d(u)+f(n,r,s,x)/s\stackrel{(\ref{eq: turan case, lower bound on diff degree sum})}{\leq}\Sigma_{r,0}(a,n+1)-\Sigma_{r,0}(a,n).
\end{align}
 Note the right-hand side of \eqref{eq: bound on deficiency turan} is the minimum degree of $T^e_{r,0}(a,n+1)=T^P_{r,0}(a,n+1)$. Since all edges in $T^P_{r,0}(a, n+1)$ have multiplicity $a$ or $(a+1)$, the integral AM-GM inequality (Proposition~\ref{prop: integral AM-GM}(i)) then implies  that 
 \[p_{G}(u)< \left(\frac{a}{a+1}\right)^{\frac{n-r(s+2)}{rs}}\Pi_{r,0}(a,n+1)/\Pi_{r,0}(a,n),\]
 as claimed. 
\end{proof}
\begin{proof}[Proof of Theorem~\ref{theorem: turan}]
As noted in~\cite{mt2}, the case $r=1$ is immediate from averaging over $s$-sets and applying the AM-GM inequality: for all integers $n\geq s$ and $a\geq 2$, $\textrm{ex}_{\Pi}(n, s, a\binom{s}{2})= a^{\binom{n}{2}}$, with equality uniquely attained by the multigraph on $[n]$ in which every edge has multiplicity $a$.

Fix therefore $r\in \mathbb{Z}_{\geq 2}$ and $s\geq r+1$. We shall prove that $\mathrm{ex}_{\Pi}\left(n, s, \Sigma_{r,0}(a,s)\right)=\Pi_{r,0}(a,n)$ for all $n\geq  2r(s+2)+r(s+2)\sqrt{s-1}$. The lower bound is immediate from Construction~\ref{construction: lower bound}, so we need only concern ourselves with the upper bound.

By averaging over all $s$-sets, one has that if $G\in \mathcal{F}(n, s, \Sigma_{r,0}(a,s) )$ then
\[e(G)\leq \binom{n}{s}\Sigma_{r,0}(a,s)/\binom{n-2}{s-2}<(a+1)\binom{n}{2}.\]
In particular by the AM-GM inequality and a standard upper bound on the size of Tur\'an graphs, we have:
\begin{align}\label{eq: turan bound on the base ish case}
\mathrm{ex}_{\Pi}(r(s+2), s, \Sigma_{r,0}(a,s))< (a+1)^{\binom{r(s+2)}{2}}\leq \left(\frac{a+1}{a}\right)^{\frac{1}{r}\binom{r(s+2)}{2}} \Pi_{r,0}(a,r(s+2)).
\end{align}

Now suppose $G\in \mathcal{F}(n, s, \Sigma_{r,0}(a,s))$, for some integer $n\geq 2r(s+2)+r(s+2)\sqrt{s-1}$. We sequentially remove vertices of minimum product-degree until one of two things occurs:
\begin{enumerate}[(a)]	
\item we have only $n_0=r(s+2)$ vertices left;
\item we have obtained a multigraph $G'$ on $n_0$ vertices so that $r(s+2)<n_0\leq n$ and
$G'\in \cap_{2\leq t\leq s} \mathcal{F}(n_0,t, \Sigma_{r,0}(a,t))$.
\end{enumerate}
Denote by $G=G_n, G_{n-1}, G_{n-2}, \ldots , G_{n_0}$ the sequence of induced subgraphs of our multigraph $G$ obtained by this procedure, where for each $i< n$, $G_i$ is the multigraph on $i$ vertices obtained from $G_{i+1}$ by removing an arbitrarily chosen vertex of minimum product-degree in $G_{i+1}$. For each $i$, note that $G_i\in \cap_{s_i\leq t\leq s}\mathcal{F}(i, t, \Sigma_{r,0}(a,t))$ for some integer $s_i$ so that $ 2\leq s_i \leq s$. By Lemma~\ref{lemma: turan degree-removal} (which we can apply since $n_0\geq r(s+2)$), we have that for $i\geq n_0$, 
\begin{align*}
P(G_{i+1} )\leq \Bigl(\frac{a}{a+1}\Bigr)^{\frac{i-r(s_i+1)}{r(s_i-1)} }\frac{\Pi_{r,0}(a,i+1)}{\Pi_{r,0}(a,i)} P(G_{i})
	\leq \Bigl(\frac{a}{a+1}\Bigr)^{\frac{i-r(s+1)}{r(s-1)}}\frac{\Pi_{r,0}(a,i+1)}{\Pi_{r,0}(a,i)}P(G_{i}),
\end{align*}
whence we have
\begin{align} \label{eq: upper bound on P(Gn) Turan case}
 P(G_n)\leq \Bigl(\frac{a}{a+1}\Bigr)^{\sum_{i=n_0}^{n-1}\frac{i-r(s+1)}{r(s-1)}} \frac{\Pi_{r,0}(a,n)}{\Pi_{r,0}(a,n_0)}P(G_{n_0}).
 \end{align}
If (a) occurs, then $n_0=r(s+2)$ and, by our lower bound on $n$ we have
\begin{align*}
\sum_{i=n_0}^{n-1}\frac{i-r(s+1)}{r(s-1)}> \frac{1}{2r(s-1)}\Bigl(n-2r(s+2)\Bigr)^2>\frac{1}{r}\binom{r(s+2)}{2}.
\end{align*}
Combining this inequality with \eqref{eq: upper bound on P(Gn) Turan case} and \eqref{eq: turan bound on the base ish case}, we obtain
\begin{align*}
P(G_n)< \Bigl(\frac{a}{a+1}\Bigr)^{\frac{1}{r}\binom{r(s+2)}{2}} \frac{\Pi_{r,0}(a,n)}{\Pi_{r,0}(a,n_0)}\mathrm{ex}_{\Pi}(n_0, s, \Sigma_{r,0}(a,s))\leq \Pi_{r,0}(a,n),
\end{align*}
as desired, and with a strict inequality. 

On the other hand, if (b) occurs and we have that $G_{n_0}\in \cap_{2\leq t\leq s} \mathcal{F}(n_0,t, \Sigma_{r,0}(a,t))$, then $G_{n_0}$ is a multigraph on at least $r(s+2)$ vertices in which all edges have multiplicity at most $a+1$ and for which the subgraph of edges  with multiplicity exactly $a+1$ is $K_{r+1}$-free. It follows from Tur\'an's theorem that $G_{n_0}$ has at most as many edges of multiplicity $a+1$ as a multigraph from $\mathcal{T}_{r,0}(a,n_0)$, whence $P(G_{n_0})\leq \Pi_{r,0}(a, n_0)$. Combining this with \eqref{eq: upper bound on P(Gn) Turan case}, we get that
\begin{align*}
P(G_n)\leq \Bigl(\frac{a}{a+1}\Bigr)^{\sum_{i=n_0}^{n-1}\frac{i-r(s+1)}{r(s-1)}} \frac{\Pi_{r,0}(a,n)}{\Pi_{r,0}(a,n_0)}P(G_{n_0})\leq \Pi_{r,0}(a,n),
\end{align*}
with equality achieved if and only if $n_0=n$ and $G=G_n \in \mathcal{T}_{r,0}(a,n)$. This concludes the proof of the theorem.
\end{proof}

\section{Proof of the Mubayi--Terry conjecture}
\subsection{Overview of the proof}\label{sec71}
In this section we prove Theorem~\ref{theorem: mubayi Terry conjecture}, thereby resolving the conjecture of Mubayi and Terry. We follow an essentially inductive proof strategy. First of all we establish the statement of Theorem~\ref{theorem: mubayi Terry conjecture} holds in the case $n=6$. Next, we show that if $G\in \mathcal{F}(n+1, 4, \binom{4}{2}a +3)$, then either $G$ contains a vertex with product-degree at most $\Pi_{2,1}(a,n+1)/\Pi_{2,1}(a,n)$ or $G \in  \mathcal{F}(n+1, 3, \binom{3}{2}a +2)\cap \mathcal{F}(n+1, 2, \binom{2}{2}a +1)$. In the former case, we remove a vertex with low product-degree and repeat our argument. In the latter case, we show we may modify $G$ to obtain a multigraph $G'$ satisfying certain properties.

\begin{proposition}\label{prop: modifying G to get the clique structure}
Let $G\in \bigcap_{2\leq t \leq 4} \mathcal{F}(n+1, t, \binom{t}{2}a+t-1)$. Then there exists a multigraph $G'$ satisfying the following three properties:
\begin{enumerate}[(a)]
	\item $G'\in \bigcap_{2\leq t \leq 4} \mathcal{F}(n+1, t, \binom{t}{2}a+t-1)$,
	\item $P(G')\geq P(G)$,
	\item  $V(G')$ can be partitioned into disjoint sets $X_1,X_2,\ldots ,X_m$ for some $m\leq n$ such that, for each $i$, $G'[X_i]$ is a clique of edges with multiplicity $a-1$, and for each pair $i,j$ with $i\neq j$, all edges between $X_i$ and $X_j$ have the same multiplicity in $G'$, and this multiplicity is either $a$ or $a+1$.
\end{enumerate} 
\end{proposition}

The modified multigraph $G'$ can be viewed as a vertex-weighted subgraph $H$ of $K_m$, with the vertex $i$ receiving weight $\vert X_i\vert$ and the edge $ij$ included in $H$ if the edges between $X_i$ and $X_j$ have multiplicity $a+1$ in $G'$, and omitted if they have multiplicity $a$. We will perform some optimisation to show that either $G'$ contains a vertex with product-degree at most $\Pi_{2,1}(a,n+1)/\Pi_{2,1}(a,n)$ (provided $n$ is sufficiently large) or  $H$ is acyclic. In the former case we remove a low product-degree vertex and repeat our argument. In the  latter case, we perform some simple local modifications to show there exists $G''$ such that $P(G')\leq P(G'')$ and $G'' \in \mathcal{T}_{2,1}(a,n)$, from which the theorem follows.

\subsection{The base cases and some useful properties}
We first establish the value of $\mathrm{ex}_{\Pi}(n, 4, \binom{4}{2}a +3)$ for
$n=4,5,6$.
\begin{proposition}[Base cases]\label{proposition: small base case} The following hold for all integers $a\geq 2$:
	\begin{enumerate}[(i)]
		\item $\mathrm{ex}_{\Pi}(4, 4, \binom{4}{2}a +3) = \Pi_{2,1}(a,4)=a^3(a+1)^3$,
		\item $\mathrm{ex}_{\Pi}(5, 4, \binom{4}{2}a +3) = a^5(a+1)^5> \Pi_{2,1}(a,5)=(a-1)a^3(a+1)^6$,
		\item $\mathrm{ex}_{\Pi}(6, 4, \binom{4}{2}a +3)= \Pi_{2,1}(a,6)=(a-1)a^6(a+1)^{8}$.
	\end{enumerate}
\end{proposition}
\begin{proof}
The integral version of the AM-GM inequality (Proposition~\ref{prop: integral AM-GM}(i)) immediately implies that $ \mathrm{ex}_{\Pi}(4, 4, \binom{4}{2}a +3)\leq a^3(a+1)^3$. Partitioning $[4]$ into $V_0=\{1\}$ and $V_1=\{2,3,4\}$ then shows  that $a^3(a+1)^3 \leq \Pi_{2,1}(a,4) \leq \mathrm{ex}_{\Pi}(4, 4, \binom{4}{2}a +3)$, establishing (i).

For (ii), it easy to check that among all graphs $G \in \mathcal{T}_{2,1}(a,5)$, the quantity $P(G)$ is maximised when $V_{0} = \{1,2\}$ and $V_{1}= \{3,4,5\}$, which shows that $\Pi_{2,1}(a, 5)= (a-1)a^3(a+1)^6$.  It is also easy to check that this quantity is strictly less than $a^5(a+1)^5$ when $a \geq 2$.  In the other direction, let $G$ be a product-extremal graph in $\mathcal{F}(5,4,\binom{4}{2}a+3)$. By averaging over all $4$-sets, we see that $e(G)\leq \frac{5}{3}\left(\binom{4}{2}a+3\right)= \binom{5}{2}a+5$. The integral AM-GM inequality (Proposition~\ref{prop: integral AM-GM}(i)) then gives $P(G)\leq a^5(a+1)^5$.
On the other hand, consider the $5$-vertex multigraph $G'$ whose edges of multiplicity $a+1$ form a $5$-cycle $C_5$, with all other edges having multiplicity $a$. Then $P(G')= a^5(a+1)^5$ and $G' \in \mathcal{F}(5,4,\binom{4}{2}a+3)$; so indeed,
$\mathrm{ex}_{\Pi}(5, 4, \binom{4}{2}a +3) = a^5(a+1)^5$.

For (iii), consider the graph $G \in \mathcal{T}_{2,1}(a,6)$ with $V_{0} = \{1,2\}$ and $V_{1} = \{3,4,5,6\}$.  This graph shows that $\Pi_{2,1}(a,6) \geq (a-1)a^{6}(a+1)^{8}$.  In the other direction, let $G$ be a product-extremal graph in $\mathcal{F}(6,4,\binom{4}{2}a+3)$. By averaging over all $4$-sets, we see that $e(G)\leq \left\lfloor\frac{15}{6}\left(\binom{4}{2}a+3\right)\right\rfloor= \binom{6}{2}a+\left\lfloor \frac{15}{2}\right \rfloor = \binom{6}{2}a + 7$. If $e(G)\leq \binom{6}{2}a + 6$, then by the integral version of the AM-GM inequality, $P(G)\leq a^9 (a+1)^6< (a-1)a^6(a+1)^8$. Suppose now that $e(G)= \binom{6}{2}a + 7$. If $G$ contains at least one edge of multiplicity at most $a-1$, then by Proposition~\ref{prop: integral AM-GM}(ii), $P(G)\leq (a-1)a^{6}(a+1)^8$.

On the other hand, suppose all edges of $G$ have multiplicity at least $a$. Since $G\in \mathcal{F}(6, 4, \binom{4}{2}a+3)$, this implies the edges in $G$ have multiplicity between $a$ and $a+3$. If every edge of $G$ has multiplicity either $a$ or $a+1$, then $G$ has exactly $7$ edges of multiplicity $a+1$. 
This, together with the property that no $4$-set of vertices in $G$ contains four edges of multiplicity $a+1$, implies that
 there is a triangle $T$ of edges with multiplicity $a+1$ in $G$.\footnote{Indeed, suppose no such triangle exists. Then there is a path $P$ of length $2$ in $G$ with both edges of 
multiplicity $a+1$. Further, (i) there must be  a vertex outside of $P$ that sends at
least two edges of multiplicity $a+1$ to $P$  or (ii) $V(G)\backslash P$ itself spans a path $P'$ of length $2$ in $G$ with both edges of 
multiplicity $a+1$, and there is a perfect matching between the vertices of $P$ and $P'$ consisting of edges of multiplicity $a+1$. In both cases we obtain a $4$-set containing at least $4$ edges of multiplicity $a+1$, a contradiction.}
 Since $G\in \mathcal{F}(6, 4, \binom{4}{2}a+3)$, our assumptions on $G$ imply every edge from $V(G)\setminus V(T)$ to $V(T)$ has multiplicity exactly $a$. But then $G$ can have at most $6$ edges of multiplicity $a+1$, a contradiction.  If $G$ contains an edge of multiplicity $a+3$, then every other edge has multiplicity exactly $a$, and $P(G)= (a+3)a^{14}< (a-1)a^6(a+1)^8$ for $a\geq 2$. 

Finally if $G$ contains an edge  $e_0$ of multiplicity $a+2$, then every other edge has multiplicity $a$ or $a+1$. Suppose no edge of multiplicity $a+1$ is incident with $e_0$. Then $P(G)\leq w(e_0)a^8 \mathrm{ex}_{\Pi}(4,4,\binom{4}{2}a+3)= a^{11}(a+1)^3(a+2) <(a-1)a^6(a+1)^8$.   On the other hand, suppose there is some vertex $v$ sending an edge of multiplicity $a+1$ to one of $e_0$'s endpoints. Then this vertex is unique, it sends exactly one such edge into $e_0$, and every other vertex sends only edges of multiplicity $a$ into $e_0\cup\{v\}$. Thus $P(G)\leq P(G[e_0\cup\{v\}]) a^9(a+1)^3= a^{10}(a+1)^4(a+2)<(a-1)a^6(a+1)^8$ for $a\geq 2$.  

Altogether, this shows that $\mathrm{ex}_{\Pi}(6,4, \binom{4}{2}a+3)\leq (a-1)a^6(a+1)^8$, which gives the result. 
\end{proof}

The next results together show that if $G$ has a `heavy triangle' or a `heavy edge', then $G$ must contain a vertex of relatively small product-degree. This can be thought of as a specialised and more precise version of Theorem~\ref{theorem: step down}.
\begin{lemma}[Heavy triangles give low product-degree vertices]\label{lemma: heavy triangles give low-deg vces}
Let $G\in \mathcal{F}(n+1, 4, \binom{4}{2}a+3)$, with $n\geq 6$. Then either every triangle in $G$ has edge-sum at most $\binom{3}{2}a+2$ or $G$ contains a vertex of product-degree at most $(a+1)^2a^{n-2}$.	
\end{lemma}
\begin{proof}
Suppose $G$ contains a triangle $T$ whose edge-sum is equal to $\binom{3}{2}a+3+ 3x$, for some non-negative rational $x$. Then, since $G\in \mathcal{F}(n+1, 4, \binom{4}{2}a+3)$, every vertex $v\in V(G)\setminus V(T)$ sends at most $3a -3x$ edges (counting multiplicities) into $V(T)$. Thus by the AM-GM inequality, the product of the product-degrees of the vertices of $T$ is at most $(a+1+x)^6(a-x)^{3(n-2)}$. Simple calculus shows that for $n\geq 6$, this expression is strictly decreasing in $x$. By averaging, this implies that some vertex in $T$ has product-degree at most $(a+1)^2a^{n-2}$.
\end{proof}
\begin{lemma}[Heavy edges give low product-degree vertices]\label{lemma: heavyedges give low-deg vces}
	Let $G\in \mathcal{F}(n+1, 3, \binom{3}{2}a+2)$, with $n\geq 6$. Then either every edge in $G$ has multiplicity at most $a+1$ or $G$ contains a vertex of product-degree at most $(a+2)a^{n-1}$.	
\end{lemma}
\begin{proof}
Suppose $G$ contains an edge $e_0$ of multiplicity $a+2+2x$, for some non-negative rational $x$. Then, since $G\in \mathcal{F}(n+1, 3, \binom{3}{2}a+2)$, every vertex $v\in V(G)\setminus e_0$ sends at most $2a -2x$ edges (counting multiplicities) into $e_0$. Thus by the AM-GM inequality, the product of the product-degrees of the vertices of $e_0$ is at most $(a+2+2x)^2(a-x)^{2(n-1)}$. Simple calculus shows that for $n\geq 6$, this expression is strictly decreasing in $x$. By averaging, this implies that some vertex in $e_0$ has product-degree at most $(a+2)a^{n-1}$.		
\end{proof}

\subsection{Proof of Proposition~\ref{prop: modifying G to get the clique structure}}
\begin{proof}[Proof of Proposition \ref{prop: modifying G to get the clique structure}]
We define a sequence of multigraphs $G_i$ satisfying (a)--(b) as follows. Set $G_0:=G$. At each stage $i\geq 0$, if $G_i$ satisfies (a)--(c), we are done. Otherwise $G_i$ contains either an edge $v_1v_2$ of multiplicity strictly less than $a-1$ or an edge $v_1v_2$ of multiplicity exactly $a-1$ and at least one vertex $v_3\in V(G)\setminus\{v_1, v_2\}$ such that $w(v_1v_3)\neq w(v_2v_3)$. In either case, assume without loss of generality that $p_{G_i}(v_1)\leq p_{G_i}(v_2)$. Then we define a new multigraph $G_{i+1}$ on the same vertex set by setting $w_{G_{i+1}}(v_1v):=w_{G_{i}}(v_2v)$ for all $v\in V(G)\setminus\{v_1, v_2\}$; $w_{G_{i+1}}(v_1v_2):=a-1$; $w_{G_{i+1}}(uv):=w_{G_i}(uv)$ for all other pairs $uv$.

We claim that $G_{i+1}$ satisfies (a)--(b).	Indeed, 
\begin{align*}
P(G_{i+1})&=P(G_i) \frac{(a-1)}{w_{G_i}(v_1v_2)} \frac{p_{G_i}(v_2)}{p_{G_i}(v_1)}\geq P(G_i),
\end{align*}
so as $P(G_i)\geq P(G)$, (b) holds. As $G_{i} \in \mathcal{F}(n+1, 2, a+1)$, it is clear that $G_{i+1} \in \mathcal{F}(n+1, 2, a+1)$.  It is also clear that for any set of 3 vertices (respectively set of 4 vertices) that does not include both $v_{1}$ and $v_{2}$, the sum of the multiplicities of the edges spanning these vertices is at most ${3 \choose 2} a + 2$ (respectively ${4 \choose 2} a + 3$) in $G_{i+1}$.  Given any $v_{4} \in V(G) \setminus \{v_1,v_2\}$, we know that $w_{G_{i+1}}(v_{1},v_{2}) = a-1$ and $w_{G_{i+1}}(v_{1},v_{4}) = w_{G_{i+1}}(v_{2},v_{4}) \leq a+1$.  Thus the set $\{v_1,v_2,v_{4}\}$ spans at most $3a+1$ edges, and so $G_{i+1} \in \mathcal{F}(n+1, 3, {3 \choose 2}a+2)$. 

 Similarly, given any distinct $v_{4}, v_{5} \in V(G) \setminus \{v_1,v_2\}$, we know that $w_{G_{i+1}}(v_{1},v_{2}) = a-1$, $w_{G_{i+1}}(v_{1},v_{4})\leq a+1$, $w_{G_{i+1}}(v_{1},v_{5}) \leq a+1$, and the set $\{v_2,v_4,v_{5}\}$ spans at most ${3 \choose 2}a + 2$ edges (as $G_{i} \in \mathcal{F}(n+1, 3, {3 \choose 2}a+2)$).  Thus the set $\{v_1,v_2,v_{4}, v_5\}$ spans at most ${4 \choose 2}a+3$ edges, and so $G_{i+1} \in \mathcal{F}(n+1, 4, {4 \choose 2}a+3)$.  Therefore $G_{i+1} \in \bigcap_{2\leq t \leq 4}\mathcal{F}(n+1, t, \binom{t}{2}a+t-1)$ and so $G_{i+1}$ satisfies (a).

Observe that  for all $j\geq i+1$ and $v_3\in V(G)\setminus\{v_1,v_2\}$, we have $w_{G_j}(v_1v_2)= w_{G_{i+1}}(v_1v_2)=a-1$ and $w_{G_j}(v_1v_3)=w_{G_j}(v_2v_3)$.
Thus, after $T\leq \binom{n}{2}$ steps all edges $v_1v_2$ of multiplicity $a-1$ will satisfy this condition (and there will be no edges of multiplicity less than $a-1$); it is easy to see that this immediately implies $G_T$ satisfies (c).
 Setting $G':=G_T$ concludes the proof.	
\end{proof}

\subsection{Further auxiliary results}
Given a multigraph $G'$ satisfying Proposition~\ref{prop: modifying G to get the clique structure}(c), we define an auxiliary vertex-weighted graph $H$ on $[m]$ as outlined in Section~\ref{sec71}: we assign to the vertex $i\in [m]$ the weight $\alpha_i:=\vert X_i\vert $, and we include the edge $ij$ in $E(H)$ if and only if edges from $X_i$ to $X_j$ have multiplicity $a+1$ in $G'$.
\begin{proposition}\label{prop: acyclic H is good}
Let $G'$ be a multigraph on $n$ vertices satisfying (c), and  let $H$ be the auxiliary vertex-weighted graph defined from $G'$ as above.  If $H$ is acyclic, then there exists $G'' \in \mathcal{T}_{2,1}(a,n)$ such that $P(G')\leq P(G'')$.
\end{proposition}
\begin{proof}
Suppose $H$ is acyclic. Then $H$ is a forest. For each connected component $C$ of $H$, let $v_C$ denote a vertex of $C$ with maximum weight, and orient each edge of $C$ towards $v_C$; so every edge has an orientation and we have maximum out-degree $1$. Further, let $v_0$ denote a vertex of $H$ with the maximum weight.  Given two vertices $u,v \in V(H)$, we write $\overrightarrow{uv} \in E(H)$ to mean $H$ has an edge oriented from $u$ to $v$ in our orientation. 

We define a new vertex weighted graph $H'$ from our orientation of $H$ as follows.  The two graphs $H$ and $H'$ have the same vertex set and the same vertex-weighting, and $H'$ has the edge set $E(H') := \{uv_{0}:\overrightarrow{uv} \in E(H) \}$.  We can now reverse-engineer our construction of an auxiliary vertex-weighted graph, and obtain from $H'$ a multigraph $G'_2$, by assigning to each vertex $v_i$ in $H'$ with (positive integer) weight $\alpha_i$  a  set $V_i$ of $\alpha_i$ vertices, and setting multiplicities of a pair $uv$ in $G'_2$ to be equal to $a-1$ if $uv$ lies inside some $V_i$; equal to $a+1$ if $u\in V_i$ and $v\in V_j$ for some $ij \in E(H')$;  equal to $a$ otherwise.

Now each oriented edge $\overrightarrow{uv}$ in $H$ was replaced in $H'$ by the edge $uv_0$, where the weight of $v_0$ is at least that of $v$. Since the out-degrees in our orientation of $H$ are bounded above by $1$, it follows that going from $G'$ to $G'_2$, the number of edges of multiplicity $a+1$ has not decreased, while the number of edges of multiplicity $a-1$ has remained constant. This immediately implies $P(G')\leq P(G'_2)$.

Finally consider replacing all edges between $V_{0}$ and $V(G')\setminus V_{0}$ in $G'_2$  with edges of multiplicity $a+1$, and all edges between two vertices in $V(G')\setminus V_{0}$ with edges of multiplicity $a$.  Let us call this new graph $G''$.  Note that every time we replaced an edge of $G'_2$, we never decreased its multiplicity.  Thus we have that $P(G'_2)\leq P(G'')$.  Moreover, it is easy to see that $G''\in \mathcal{T}_{2,1}(a,n)$, as desired, and so we have proven the result.
\end{proof}

\begin{lemma}[Cycles in $H$ give low product-degree vertices]\label{lemma: H is not acyclic implies low-deg vces}
	Let $G'$ be a multigraph on $n+1$ vertices satisfying both (a) and (c) of Proposition \ref{prop: modifying G to get the clique structure}, and let $H$ be the auxiliary graph defined above.  If $H$ is not acyclic, then $G'$ contains a vertex with product-degree at most $a^{\frac{4n-6}{5}}(a+1)^{\frac{n+6}{5}}$. 
\end{lemma}
\begin{proof}
Suppose $H$ is not acyclic. Let $t$ denote the girth of $G$. By relabelling vertices of $V(H)=[m]$ if necessary, we may assume that $[t]$ induces a shortest cycle $C$ in $H$, with $12,23, \ldots, (t-1)t, t1$ all being edges of $H$. Since $C$ is a shortest cycle, no other edge is present in $H[t]$. Since $G'$ satisfies (a), it immediately follows that every $3$-set in $H$ spans at most $2$ edges of $H$ and every $4$-set in $H$ spans at most $3$ edges of $H$. In particular, $H$ is both $K_3$-free and $C_4$-free, and its girth $t$ satisfies $t\geq 5$.  This further implies that every vertex $v\in [m]\setminus [t]$ sends at most one edge of $H$ into $[t]$ --- since otherwise, $H$ would contain a strictly shorter cycle.

Partition $[m]\setminus [t]$ into $Y_i:=\{v\in [m]\setminus [t]:\ iv\in E(H)\}$, $1\leq i \leq t$ and $Y_0:=\left([m]\setminus [t] \right)\setminus\left(\cup_{i=1}^t Y_i\right)$.
For each $i \in \{0,\dots, t\}$, write $Z_i:= \cup_{j \in Y_i} X_j$.
 For each vertex $i\in [t]$, choose a representative vertex $x_i\in X_i$ in $V(G')$. Further, write $N$ for $\sum_{i=1}^t \vert X_i \vert$.  Then, adopting the convention that the indices of the $X_i$'s are considered cyclically so that $X_{0}:=X_t$ and $X_{t+1}:=X_1$, we have that for each  $1\leq i\leq t$,
\begin{align*}
p_{G'}(x_i)= (a-1)^{\vert X_i\vert -1} a^{n+1- \vert X_i\vert -\vert X_{i-1}\vert - \vert X_{i+1}\vert - \vert Z_i\vert }(a+1)^{\vert X_{i-1}\vert +\vert X_{i+1}\vert  + \vert Z_i\vert}.
\end{align*}
Taking logarithms and averaging, we have:
\begin{align*}
\frac{1}{t}\sum_{i=1}^t \log p_{G'}(x_i) &= \frac{1}{t}\Bigl(\bigl(N-t\bigr)\log (a-1)+    \bigl(t(n+1) - 3N- (n+1-N-\vert Z_0\vert)\bigr) \log (a)\Bigr.\\
& \Bigl.\quad  + \bigl(2N+ (n+1-N-\vert Z_0\vert)\bigr)\log (a+1) \Bigr)\\
& \leq \frac{1}{t}\Bigl( (N-t)\log (a-1) + \left((t-1)(n+1)-2N\right)\log (a) + \left(n+1+N\right)\log (a+1)\Bigr). 
\end{align*}
The upper bound above is a linear function of $N$, with leading coefficient $\frac{1}{t}\log \left(\frac{(a-1)(a+1)}{a^2}\right)<0$. Since $N\geq t$, it follows that the average of the logarithms of the $p_{G'}(x_i)$ is at most
\begin{align*}
p_t= \frac{1}{t}\Bigl(   \left((t-1)n- t-1\right)\log (a) + \left(n+t+1\right)\log(a+1) \Bigr)= \log \left(a^{n-1}(a+1)\right) + \frac{n+1}{t}\log \left(\frac{a+1}{a}\right) .
\end{align*}
Since $p_t$ is a decreasing function of $t$ and $t\geq 5$, we conclude that one of the $x_i$ has product-degree at most
$e^{p_5}=a^{n-1-\frac{n+1}{5}} (a+1)^{1+\frac{n+1}{5}} = a^{\frac{4n-6}{5}} (a+1)^{\frac{n+6}{5}}$, as claimed.
\end{proof}

The next lemma gives us rather precise information on the vertex class sizes of 
an element from $T^P_{2,1}(a,n)$.
\begin{lemma}\label{lemma: size of partition in T^P_{2,1}(a,n)}
Let 	$a,n\in \mathbb{Z}_{\geq 2}$, $G\in T^P_{2,1}(a,n)$ and let $V_0\sqcup V_1$ be the canonical partition of $V(G)$. Then 
\[\vert V_0\vert \in [(n-1)x_{2\star}, (n-1)x_{2\star}+1]\]
where $x_{2\star}=x_{2\star}(a,1)= \log \left(\frac{a+1}{a}\right)/\log \left(\frac{(a+1)^2}{a(a-1)}\right)$ is the function defined in~\eqref{eq: def of xstar}.
\end{lemma}
\begin{proof}
Set $x:=\vert V_0\vert$.  Observe that for $n\geq 2$, we have $1\leq x\leq n-1$ (since otherwise $V_0=V(G)$ or $V_1=V(G)$ and $G$ trivially fails to maximise the value of $P(G)$ over multigraphs $G\in \mathcal{T}_{2,1}(a,n)$). Since $G\in T^P_{2,1}(a,n)$, we cannot increase $P(G)$ by switching one vertex from $V_1$ to $V_0$, or vice-versa.  This implies the following inequalities:
\begin{align*}
\left(\frac{a-1}{a+1} \right)^x \left(\frac{a+1}{a}\right)^{n-x-1}\leq 1&& \textrm{ and } &&\left(\frac{a+1}{a-1}\right)^{x-1}\left(\frac{a}{a+1}\right)^{n-x}\leq 1.
\end{align*}
Taking logarithms, rearranging terms and combining the inequalities yields
\[  (n-1)x_{2\star} \leq x \leq nx_{2\star}+ \frac{\log \left(\frac{a+1}{a-1}\right)}{\log \left(\frac{(a+1)^2}{a(a-1)}\right)}= (n-1)x_{2\star}+1 . \]	
\end{proof}

\begin{corollary}\label{corollary: ratio Pi_{2,1}(a, n+1)/Pi_{2,1}(a,n)}
For any $n\geq 2$, $\Pi_{2,1}(a, n+1)/\Pi_{2,1}(a,n) \geq \left(a^{1-x_{2\star}}\left(a+1\right)^{x_{2 \star}}\right)^{n}$.
\end{corollary}
\begin{proof}
Lemma~\ref{lemma: size of partition in T^P_{2,1}(a,n)} implies that there exists a nested sequence of multigraphs $\left(G_n\right)_{n \geq 2}$ with	$V(G_n)=[n]$, $G_n[[m]]=G_m$, $p_{G_n}(n)=\min_{i \in [n]} p_{G_n}(i)$  and $G_n \in T^P_{2,1}(a,n)$. Thus to bound 
 the ratio $\Pi_{2,1}(a, n+1)/\Pi_{2,1}(a,n)= P(G_{n+1})/P(G_n)$, it suffices to bound $p_{G_{n+1}}(n+1)$.  Let $V_0\sqcup V_1$ be the canonical partition of $G_{n+1}$, and set $x_{n+1}:=\vert V_0\vert$. We have that 
 \begin{equation}
 \frac{\Pi_{2,1}(a, n+1)}{\Pi_{2,1}(a,n)}
=p_{G_{n+1}}(n+1)
=\min\left((a-1)^{x_{n+1}-1}(a+1)^{n+1-x_{n+1}}, a^{n-x_{n+1}}
(a+1)^{x_{n+1}}\right).
 \end{equation}
 Note that $(a-1)^{x_{n+1}-1}(a+1)^{n+1-x_{n+1}}$ is a decreasing function of $x_{n+1}$, while $(a+1)^{x_{n+1}}a^{n-x_{n+1}}$ is an increasing function of $x_{n+1}$.  Thus, by Lemma \ref{lemma: size of partition in T^P_{2,1}(a,n)}, we have
 \begin{equation}
 \frac{\Pi_{2,1}(a, n+1)}{\Pi_{2,1}(a,n)} \geq \min\left((a-1)^{ x_{2\star}n}(a+1)^{(1- x_{2\star})n}, a^{(1- x_{2\star})n}(a+1)^{ x_{2\star}n}\right).
 \end{equation}
By (\ref{eq: x_star property}) we know that both of the terms in this minimum are equal, which gives the result.
\end{proof}

\subsection{Proof of Theorem~\ref{theorem: mubayi Terry conjecture}}
We are now ready to prove the Mubayi--Terry conjecture.
\begin{proof}[Proof of Theorem~\ref{theorem: mubayi Terry conjecture}]
	Suppose $n\geq 30$. Let $G=G_{n} \in \mathcal{F}(n, 4, \binom{4}{2}a +3)$. If $G_n \notin  \mathcal{F}(n, 3, \binom{3}{2}a +2)$, then by Lemma~\ref{lemma: heavy triangles give low-deg vces} we obtain  a vertex in $G_n$ of product-degree at most $(a+1)^2a^{n-3}$. We remove this vertex from $G_n$ to obtain $G_{n-1}$ and repeat the procedure with $G_{n-1}$ until we obtain a multigraph $G_{n'}$ on $n'$ vertices such that either $n'=6$ or $G_{n'} \in  \mathcal{F}(n', 3, \binom{3}{2}a +2)$.

In the former case, we set $n_0:=n'=6$ and note that $P(G_{n_0})\leq \Pi_{2,1}(a,n_0)$ by Proposition~\ref{proposition: small base case}(iii).	In the latter case, if $G_{n'} \notin  \mathcal{F}(n', 2, \binom{2}{2}a +1)$, then by Lemma~\ref{lemma: heavyedges give low-deg vces} we may find  a vertex in $G_{n'}$ of product-degree at most $(a+2)a^{n'-2}$. We remove this vertex from $G_{n'}$ to obtain $G_{n'-1}$ and repeat the procedure with $G_{n'-1}$ until we obtain a multigraph $G_{n''}$ on $n''$ vertices such that either $n''=6$ or $G_{n''} \in  \mathcal{F}(n'', 2, \binom{2}{2}a +1)$.

In the former case, we set $n_0:=n''=6$ and note that $P(G_{n_0})\leq \Pi_{2,1}(a,n_0)$ by Proposition~\ref{proposition: small base case}(iii).  In the latter case, $G_{n''} \in \cap_{2\leq t\leq 4} \mathcal{F}(n'', t, \binom{t}{2}a + t-1)$.  By Proposition~\ref{prop: modifying G to get the clique structure}, there exists a multigraph $G'_{n''}$ on $n''$ vertices such that $P(G_{n''})\leq P(G'_{n''})$ and such that $G'_{n''}$ satisfies  properties (a) and~(c) of Proposition \ref{prop: modifying G to get the clique structure}. We replace $G_{n''}$ by $G'_{n''}$, and may then define the auxiliary graph $H=H_{n''}$ for $G_{n''}$.

If $H$ is acyclic, then we set $n_0:=n''$ and note that $P(G_{n_0})\leq \Pi_{2,1}(a,n_0)$ by Proposition~\ref{prop: acyclic H is good}. Otherwise by Lemma~\ref{lemma: H is not acyclic implies low-deg vces} we may find a vertex in $G_{n''}$ of product-degree at most $a^{\frac{4n''-10}{5}}(a+1)^{\frac{n''+5}{5}}$. We remove this vertex from $G_{n''}$ to obtain $G_{n''-1}$ and repeat the procedure with $G_{n''-1}$ until we obtain a multigraph $G_{n_0}$ where either $n_0=6$ or the auxiliary graph $H_{n_0}$ for $G_{n_0}$ is acyclic. In either case, we have $P(G_{n_0})\leq \Pi_{2,1}(a, n_0)$.

Summing up the above, we have obtained from $G_n$ a multigraph $G_{n_0}$ on $n_0$ vertices where $6\leq n_0 \leq n$ and with $P(G_{n_0})\leq \Pi_{2,1}(a,n_0)$ by removing, for each stage $i$  (where $n_0< i \leq n$), a vertex with product-degree at most
\[\max\left(a^{i-2}(a+2), a^{i-3}(a+1)^2, a^{\frac{4i-10}{5}}(a+1)^{\frac{i+5}{5}}\right)=a^{\frac{4i-10}{5}}(a+1)^{\frac{i+5}{5}},\]
from $G_{i}$ and by possibly modifying the multigraph $G_i$ in ways that do not decrease the product of its edge multiplicities $P(G_i)$. Thus in particular we have
\begin{align}\label{eq: upper bound on P(G_n)}
P(G_n)\leq \Pi_{2,1}(a,n_0)\prod_{i=n_0+1}^{n}  a^{\frac{4i-10}{5}}(a+1)^{\frac{i+5}{5}}=\Pi_{2,1}(a,n_0)\prod_{i=n_0}^{n-1}  a^{\frac{4i-6}{5}}(a+1)^{\frac{i+6}{5}}. 
\end{align}
If $n_0=n$, then we are done. On the other hand if $n_0<n$, then applying Corollary~\ref{corollary: ratio Pi_{2,1}(a, n+1)/Pi_{2,1}(a,n)} yields
\begin{align}\label{eq: lower bound on Pi_{2,1}(a,n)}
\Pi_{2,1}(a,n) &=  \Pi_{2,1}(a,n_0) \prod_{i=n_0}^{n-1} \left(\frac{\Pi_{2,1}(a,i+1)}{\Pi_{2,1}(a,i)}\right)\geq \Pi_{2,1}(a,n_0)\prod_{i=n_0}^{n-1} a^{(1-x_{2 \star})i}(a+1)^{x_{2\star}i}.
	\end{align}
Combining \eqref{eq: upper bound on P(G_n)} and~\eqref{eq: lower bound on Pi_{2,1}(a,n)} , we get
\begin{align*}
\frac{\Pi_{2,1}(a,n)}{P(G_n)} &\geq \prod_{i=n_0}^{n-1}\left(\frac{a+1}{a}\right)^{\frac{(5x_{2\star}-1)i -6}{5}}=  \left(\frac{a+1}{a}\right)^{\frac{(n-n_0)}{5}\left((n+n_0-1)\frac{(5x_{2\star}-1)}{2}-6\right)}.
\end{align*}		
Provided $(n+n_0-1)\left(5x_{2\star}-1\right)\geq 12$, the above is greater or equal to $1$ and the theorem claim follows.  

Now by construction $n_0\geq 6$, and as shown in Proposition~\ref{prop: a-monotonicity of xstar}, $x_{2\star}=x_{2\star}(a,1)$ is an increasing function of $a$. 
Thus for $n$ satisfying 
\begin{align}\label{eq: n bound in MT conjecture}
n \geq \frac{12}{\left(5x_{2\star}(2,1)-1\right)}-5= \frac{12\log(9/2)}{\log(27/16)}-5=29.49\ldots \notag\end{align}
and all $a\in \mathbb{Z}_{\geq 2}$, we  have $\Pi_{2,1}(a,n)\geq P(G_n)$ for all $G_n \in \mathcal{F}(n, 4, \binom{4}{2}a+3)$, as claimed.		
	\end{proof}
	
\begin{remark}
One can lower the bound on $n$ in Theorem~\ref{theorem: mubayi Terry conjecture}
for $a\geq 3$. For example, the argument in the proof allows us to take $n\geq 21$ if $a=3$.
Furthermore, we  believe $\mathrm{ex}_{\Pi}(n, 4, \binom{4}{2}a+3)=\Pi_{2,1}(a,n)$ should hold for all $n \geq 6$. Our proof shows that this could only fail if there exists a multigraph $G$ on $n$ vertices for some  $ n \in \{7,\dots, 18\}$  with $P(G)> \Pi_{2,1}(a, n)$.  Indeed, for $n_0\geq 18$, $(n+n_0-1)(5x_{2\star}-1)\geq 12$ is satisfied for all $n\geq n_0$ and $a \in \mathbb{Z}_{\geq 2}$.
\end{remark}

\section{Beyond the Mubayi--Terry conjecture: proof of 
Theorems~\ref{theorem: (5,5)}--\ref{theorem: (7,9)}}
\subsection{Proof of Theorem~\ref{theorem: (5,5)}}
For the lower bound in Theorem~\ref{theorem: (5,5)}, $\mathrm{ex}_{\Pi}(n, 5, \binom{5}{2}a +5) \geq \Pi_{2,1}(a,n)$, it is enough to observe that $\binom{5}{2}a+5=\Sigma_{2,1}(a,5)$. For the upper bound, we continue to use our degree-removal strategy from the proof of Theorem~\ref{theorem: mubayi Terry conjecture}.
\begin{lemma}[Heavy $4$-sets give low product-degree vertices]\label{lemma: heavy 4-sets have low-deg vces}
	Let $G\in \mathcal{F}(n+1, 5, \binom{5}{2}a+5)$, with $n\geq 7$. Then either every $4$-set $S\subseteq V(G)$ satisfies $e(G[S])\leq \binom{4}{2}a+3$ or $G$ contains a vertex of product-degree at most $(a+1)^{\frac{n+5}{4}}a^{\frac{3n-5}{4}}$.	
\end{lemma}
\begin{proof}
	Suppose $G$ contains a $4$-set $S$ where $e(G[S])>\binom{4}{2}a+3$. We consider two subcases.

	If $e(G[S])=\binom{4}{2}a+4$, then, as $G\in \mathcal{F}(n+1, 5, \binom{5}{2}a+5)$, every vertex $v\in V(G)\setminus S$ sends at most $4a+1$ edges (counting multiplicities) into $S$. By the integral AM-GM inequality (Proposition~\ref{prop: integral AM-GM}(i)) it follows that the product of the product-degrees of the vertices of $S$  is at most 
	\[(a+1)^8a^4 (a+1)^{n-3}a^{3(n-3)}=(a+1)^{n+5}a^{3n-5}.\]
	By averaging, one of the vertices in $S$ must thus have product-degree at most $(a+1)^{\frac{n+5}{4}}a^{\frac{3n-5}{4}}$ in $G$, as desired.

	On the other hand, suppose $e(G[S])=\binom{4}{2}a+5+x$, for some non-negative integer $x$. Then every vertex $v\in V(G)\setminus S$ sends at most $4a-x$ edges (counting multiplicities) into $S$. Applying the AM-GM inequality, the product of the product-degrees of the vertices in $S$ is at most 
	\[ \left (a+\frac{5+x}{6} \right )^{12}\left (a-\frac{x}{4} \right )^{4(n-3)}. \]
	 Simple calculus shows that for $n\geq 6$, this expression is strictly decreasing in $x\geq 0$. By averaging, we deduce that $S$ contains some vertex with product-degree at most $(a+\frac{5}{6})^3a^{n-3}$, which for $n\geq 7$ is strictly less than $(a+1)^{\frac{n+5}{4}}a^{\frac{3n-5}{4}}$. 
\end{proof}

Consider the element $G \in \mathcal T_{2,1}(a,7)$ with canonical partition such that
$|V_0|=2$ and $|V_1|=5$; then $P(G)= (a-1) (a+1)^{10}a^{10}$ and so
$\Pi _{2,1}(a,7) \geq (a-1)(a+1)^{10}a^{10}$.

Next consider any $G \in \mathcal F (7,5, \binom{5}{2}+5)$.
By averaging over all $5$-sets, we see that 
$$e(G)\leq \left \lfloor \frac{\binom{7}{5} }{\binom{5}{3}} \left ( \binom{5}{2}a+5 \right ) \right \rfloor = \binom{7}{2} a+10.$$
The integral AM-GM inequality (Proposition~\ref{prop: integral AM-GM}(i)) then gives $P(G)\leq a^{11}(a+1)^{10}$.

Together, the above implies that 
\begin{align}\label{compare}
\mathrm{ex}_{\Pi}(7, 5, \binom{5}{2}a+5) \leq a^{11}(a+1)^{10} \leq \frac{a}{a-1} 
\Pi_{2,1}(a,7).
\end{align}

We are now in a position to prove Theorem~\ref{theorem: (5,5)}. 

\begin{proof}[Proof of Theorem~\ref{theorem: (5,5)}]
Let $n\geq 124$, and let $G=G_{n}$ be a multigraph in $\mathcal{F}(n, 5, \binom{5}{2}a+5)$. We produce a sequence of multigraphs as follows.

As long as $G_{i} \notin \mathcal{F}(i,4, \binom{4}{2}a+3)$ and $i\geq 8$, we remove from $G_i$ a vertex of product-degree at most $(a+1)^{\frac{i+4}{4}}a^{\frac{3i-8}{4}}$ (which exists by Lemma~\ref{lemma: heavy 4-sets have low-deg vces}) to obtain $G_{i-1}$. Let $G_t$ be the last multigraph obtained in this process.

If $t=7$,  we observe that 
\begin{align}\label{eq: bound on P(G_n)}
P(G_n)&\leq \left(\prod_{i=8}^{n}(a+1)^{\frac{i+4}{4}}a^{\frac{3i-8}{4}}\right) \mathrm{ex}_{\Pi}(7, 5, \binom{5}{2}a+5) 
	\stackrel{(\ref{compare})}{\leq }  \left(\prod_{i=7}^{n-1}(a+1)^{\frac{i+5}{4}}a^{\frac{3i-5}{4}}\right) 
	\frac{a}{a-1} \Pi_{2,1}(a,7).
\end{align}
Combining this with~\eqref{eq: lower bound on Pi_{2,1}(a,n)} (where we take $n_0=7$), we get
\begin{align}\label{eq: upper bound on ratio (5,5) case i=5}
\frac{\Pi_{2,1}(a,n)}{P(G_n)}&\geq \frac{a-1}{a} \prod_{i=7}^{n-1}\left(\frac{a+1}{a}\right)^{\frac{(4x_{2\star}-1)i -5}{4}}
= \frac{(a-1)}{a}\left(\frac{a+1}{a}\right)^{\frac{(n-7)}{8}\left((4x_{2\star}-1)(n+6)-10\right)}.
\end{align}
Notice that if 
$\frac{(n-7)}{8}\left((4x_{2\star}-1)(n+6)-10\right)\geq 2$ then the right hand side of (\ref{eq: upper bound on ratio (5,5) case i=5}) is at least $1$ (since $(a-1)(a+1)^2>a^3$ for $a\geq 2$).

On the other hand if $t>7$ then $G_t\in \mathcal{F}(t,4, \binom{4}{2}a+3)$.
This allows us to follow the same argument as in the proof of \eqref{eq: upper bound on P(G_n)}. That is, there is some $n_0 \in \mathbb N$ where $6\leq n_0 \leq t$ so that
\begin{align}\label{eq10}
P(G_t)\leq \Pi_{2,1}(a,n_0)\prod_{i=n_0}^{t-1}  a^{\frac{4i-6}{5}}(a+1)^{\frac{i+6}{5}}. 
\end{align}
Thus,
\begin{align}\label{eq: bound on P(G_n) (5,5)}
P(G_n)&\leq \left(\prod_{i=t+1}^{n}(a+1)^{\frac{i+4}{4}}a^{\frac{3i-8}{4}}\right) P(G_t)\notag \\
&\stackrel{(\ref{eq10})}{\leq}   \left(\prod_{i=t}^{n-1}(a+1)^{\frac{i+5}{4}}a^{\frac{3i-5}{4}}\right) \left(\prod_{i=n_0}^{t-1}a^{\frac{4i-6}{5}}(a+1)^{\frac{i+6}{5}} \right)\Pi_{2,1}(a,n_0).
\end{align}
 Combining this with ~\eqref{eq: lower bound on Pi_{2,1}(a,n)}, and using the fact that for all $i\geq 0$
\[\frac{(5x_{2\star}-1)i-6}{5}>\frac{(4x_{2\star}-1)i-5}{4},\]
we get 
\begin{align}\label{eq: upper bound on ratio (5,5) case (4,3)}
\frac{\Pi_{2,1}(a,n)}{P(G_n)}&\geq \left(\prod_{i=t}^{n-1}\left(\frac{a+1}{a}\right)^{\frac{(4x_{2\star}-1)i -5}{4}}
\right) \left( \prod_{i=n_0}^{t-1}\left(\frac{a+1}{a}\right)^{\frac{(5x_{2\star}-1)i -6}{5}} \right) \geq  \left(\prod_{i=n_0}^{n-1}\left(\frac{a+1}{a}\right)^{\frac{(4x_{2\star}-1)i -5}{4}}
\right)\notag \\
& = \left(\frac{a+1}{a}\right)^{\frac{(n-n_0)}{8}\left((n+n_0-1)(4x_{2\star}-1) -10 \right)}.
\end{align}
Therefore \eqref{eq: upper bound on ratio (5,5) case i=5} and \eqref{eq: upper bound on ratio (5,5) case (4,3)} imply that if 
\begin{align*}
 \frac{(n-7)}{8}\Bigl((4x_{2\star}-1)(n+6)-10\Bigr)\geq 2 && \textrm{ and } && (n+n_0-1)(4x_{2\star}-1) -10 \geq 0,
\end{align*}
then $\Pi_{2,1}(a,n)/P(G_n)\geq 1$. Recall that $x_{2\star}$ is  increasing in $a$ by Proposition~\ref{prop: a-monotonicity of xstar}, so that $x_{2\star}=x_{2\star}(a,1)\geq x_{2\star}(2,1)= \log (3/2)/\log (9/2)$. Since $n_0\geq 6$, we deduce that for
\[n \geq \Bigl\lceil\max\left(123,  \frac{10+16/116}{4x_{2\star}(2,1)-1}-6, \frac{10}{4x_{2\star}(2,1)-1}-5  \right)\Bigr\rceil=124,\]
$\mathrm{ex}_{\Pi}(n, 5,\binom{5}{2}a+5)=\Pi_{2,1}(a,n)$ as claimed.
\end{proof}

\subsection{Proof of Theorems~\ref{theorem: (6,7)} and~\ref{theorem: (7,9)}}

\begin{proof}[Proof of Theorem~\ref{theorem: (6,7)}]
	This is very similar to the proof of Theorem~\ref{theorem: (5,5)}. For the lower bound, $\mathrm{ex}_{\Pi}(n, 6, \binom{6}{2}a +7) \geq \Pi_{2,1}(a,n)$, it is enough to observe that $\binom{6}{2}a+7=\Sigma_{2,1}(a,6)$. For the upper bound, we continue to use our degree-removal strategy.
\begin{lemma}[Heavy $5$-sets give low product-degree vertices]\label{lemma: heavy 5-sets have low-deg vces}
	Let $G\in \mathcal{F}(n+1, 6, \binom{6}{2}a+7)$, with $n\geq 12$. Then either every $5$-set in $G$ has edge-sum at most $\binom{5}{2}a+5$ or $G$ contains a vertex of product-degree at most $(a+1)^{\frac{n+8}{5}}a^{\frac{4n-8}{5}}$.	
\end{lemma}	
\begin{proof}
Easy modification of the proof of Lemma~\ref{lemma: heavy 4-sets have low-deg vces}.
\end{proof}

One can then follow a cruder version of the degree-removal strategy from the proof of Theorem~\ref{theorem: (5,5)} to obtain Theorem~\ref{theorem: (6,7)}. By averaging over all $6$-sets, it is easy to see that for any $G_{124}\in \mathcal{F}(124, 6, \binom{6}{2}a+7)$ the average edge multiplicity is at most $a+1$, whence by the AM-GM inequality we have 
\begin{align}\label{corollsec: bound on P(G_{124})}
P(G_{124})\leq \left(\frac{a+1}{a}\right)^{\binom{124}{2}}\Pi_{2,1}(a, 124).
\end{align}

Now, let $n\geq 503$ and let $G=G_n \in \mathcal{F}(n, 6, \binom{6}{2}a+7)$. As in the proof of Theorem~\ref{theorem: (5,5)}, sequentially remove vertices of minimum product-degree to obtain a sequence of multigraphs $G_n, G_{n-1}, G_{n-2}, \ldots$. Let $n'$ be the greatest integer $n'\leq n$ such that either $n'= 124$ or $G_{n'}\in \mathcal{F}(n',5, \binom{5}{2}a+5)$. By Lemma~\ref{lemma: heavy 5-sets have low-deg vces} and Corollary~\ref{corollary: ratio Pi_{2,1}(a, n+1)/Pi_{2,1}(a,n)}, we have
\begin{align}\label{corollsec eq: bound on P(Gn)}
P(G_n)&\leq P(G_{n'})\prod_{i=n'}^{n-1}(a+1)^{\frac{i+8}{5}}a^{\frac{4i-8}{5}}
\leq \Pi_{2,1}(a,n) \frac{P(G_{n'})}{\Pi_{2,1}(a,n')} \prod_{i=n'}^{n-1} \left(\frac{a+1}{a}\right)^{\frac{8 -(5x_{2\star}-1)i}{5}}\\ \notag
&= \Pi_{2,1}(a,n) \frac{P(G_{n'})}{\Pi_{2,1}(a,n')}\left(\frac{a+1}{a}\right)^{\left(\frac{n-n'}{5}\right)\left(8 -(5x_{2\star}-1)\left(\frac{n+n'-1}{2}\right)\right)}.
\end{align}
Now by the monotonicity in $a$ established in Proposition~\ref{prop: a-monotonicity of xstar}, $x_{2\star}=x_{2\star}(a,1)\geq \log(3/2)/\log (9/2)$. For $n \geq 503$ and $n'\geq 124$ we have
\begin{align}\label{corollsec eq: bound on xstar expression}
(5x_{2\star}-1)\left(\frac{n+n'-1}{2}\right)-8> 0 &&\textrm{ and}&& \frac{(n-124)}{5}\left((5x_{2\star}-1)\left(\frac{n+124-1}{2}\right)-8\right)> \binom{124}{2} .
\end{align}
If  $G_{n'}\in \mathcal{F}(n', \binom{5}{2}a +5)$, then as $n'\geq 124$,  Theorem~\ref{theorem: (5,5)} implies that $P(G_{n'})\leq \Pi_{2,1}(a,n')$ and \eqref{corollsec eq: bound on P(Gn)} together with the first of the bounds in~\eqref{corollsec eq: bound on xstar expression} implies $P(G_n)\leq \Pi_{2,1}(a,n)$. If on the other hand $n'=124$ then \eqref{corollsec eq: bound on P(Gn)} together with the second of the bounds in~\eqref{corollsec eq: bound on xstar expression} implies
\[P(G_n)\leq  \Pi_{2,1}(a,n) \frac{P(G_{124})}{\Pi_{2,1}(a,124)}\left(\frac{a+1}{a}\right)^{-\binom{124}{2}}\leq \
\Pi_{2,1}(a,n),\]
where the last inequality is from \eqref{corollsec: bound on P(G_{124})}. Thus in either case $P(G_n)\leq \Pi_{2,1}(a,n)$, as claimed.
\end{proof}

\begin{proof}[Proof of Theorem~\ref{theorem: (7,9)}]
This proof follows exactly that of Theorem~\ref{theorem: (6,7)}, mutatis mutandis. The analogue of Lemma~\ref{lemma: heavy 5-sets have low-deg vces} is the following.
\begin{lemma}[Heavy $6$-sets give low product-degree vertices]\label{lemma: heavy 6-sets have low-deg vces}
	Let $G\in \mathcal{F}(n+1, 7, \binom{7}{2}a+9)$. Then for all $n$ sufficiently large, either every $6$-set in $G$ has edge-sum at most $\binom{6}{2}a+7$ or $G$ contains a vertex of product-degree at most $(a+1)^{\frac{n+11}{6}}a^{\frac{5n-11}{6}}$.	
\end{lemma}	
The analogue of the key leftmost inequality in~\eqref{corollsec eq: bound on xstar expression} is
\[ \frac{6x_{2\star} -1}{2}\left(n+n'-1\right)-11>0.\]
We leave the details to the reader.	
\end{proof}
\begin{remark}\label{remark: 8 no good}
	At this point, the reader may well ask themselves why we do not keep going in this fashion and establish that $\mathrm{ex}_{\Pi}(n, 8, \Sigma_{2,1}(a,8))=\Pi_{2,1}(a,n)$ for all $a\geq 2$ and all $n$ sufficiently large. The problem here is that for $a=2$, the analogue of the leftmost inequality in~\eqref{corollsec eq: bound on xstar expression} fails:  the analogue of Lemma~\ref{lemma: heavy 6-sets have low-deg vces} gives a vertex with product-degree at most
	\[ (a+1)^{\frac{2n+8}{7}}a^{\frac{5n-8}{7}}.\]
	Following the proof of Theorem~\ref{theorem: (6,7)}, one ends up considering the quantity $(7x_{2\star}-2)/7$, which for $a=2$ is strictly negative. Thus for the pair $(8, \Sigma_{2,1}(2, 8))=(8, 2\binom{8}{2}+12)$, a set of $7$ vertices spanning at least $2\binom{7}{2}+10$ edges  does not in itself guarantee the existence of a sufficiently low product-degree vertex in a multigraph from $\mathcal{F}(n, 8, 2\binom{8}{2}+12)$.
\end{remark}

\section{The sparse case: proof of Theorem~\ref{theorem: sparse case}}
\begin{proof}[Proof of Theorem~\ref{theorem: sparse case}]
Let $s,q$ be non-negative integers with $s\geq 2$ and $q\leq 2\binom{s}{2}$. Let $n\in \mathbb{Z}_{\geq s}$ and $G$ be a multigraph from $\mathcal{F}(n, s,q)$ with $P(G)=\mathrm{ex}_{\Pi}(n,s,q)$. We bound $P(G)$ for the various values of $q$, beginning with the easiest cases. 	\newline
\noindent\textbf{Case 1: $0\leq q<\binom{s}{2}$.} In any $s$-set in $G$, there is at least one edge with multiplicity $0$. Hence $P(G)=0$, as claimed.\newline
\noindent\textbf{Case 2: $q=\binom{s}{2}$.} By averaging over all $s$-sets, the average edge-multiplicity in $G$ is at most $1$. By the AM--GM inequality, it follows that $P(G)\leq 1^{\binom{n}{2}}=1$. The $n$-vertex multigraph where every edge has multiplicity $1$ shows that the equality is attained here.\newline
\noindent\textbf{Case 3: $\binom{s}{2}<q< \binom{s}{2}+\lfloor \frac{s}{2}\rfloor $.}  Let $q' := q -\binom{s}{2}$. As $G$ is extremal, we may assume all edges have multiplicity at least $1$.  Furthermore, as $q'< \frac{s}{2}$, there are at most $q'$ edges in $G$ with multiplicity greater than $1$ as otherwise we could find a set $X$ of $s$ vertices such that $e(G[X]) > q$.  Thus, by the integral AM-GM inequality, we have that $P(G) \leq 2^{q'}$.  This maximum can be easily  attained, for example by adding a matching of $q'$ edges to a complete graph on $n$ vertices. Therefore $\mathrm{ex}_{\Pi}(n,s,q)=2^{q'}$, as claimed.
\newline
\noindent\textbf{Case 4: $\binom{s}{2}+\lfloor \frac{s}{2}\rfloor \leq q<\binom{s}{2}+ s-2$.} 	Consider the multigraph $G'$ on $[n]$ in which all edges have multiplicity $1$ except for the edges of a maximal matching. Then $P(G')= 2^{\bigl\lfloor \frac{n}{2}\bigr\rfloor}$, and any $s$ vertices span at most $\lfloor \frac{s}{2}\rfloor $ edges of multiplicity at least $2$. This shows $\mathrm{ex}_{\Pi}(n,s,q)=2^{\Omega(n)}$ in this range. For the upper bound, note that we may assume that all edges of $G$ have multiplicity at least $1$. Furthermore, consider the subgraph $G^{\geq 2}$ of edges with multiplicity at least $2$. Since no $s$-set can span more than $s-3$ edges in this graph, it follows that every component of $G^{\geq 2}$ has order at most $s-2$; moreover in each such component $C$, we have that $e(G[C])\leq \binom{\vert C\vert }{2}+ s-3$. An easy optimisation thus shows that each component of $G^{\geq 2}$ of order $t\geq 2$ can contribute at most $2^{s-3}$ to $P(G)$. There are clearly at most $n/2$ such components, whence $P(G)\leq 2^{\frac{s-3}{2}n}=2^{O(n)}$. This shows that $\mathrm{ex}_{\Pi}(n,s,q)=2^{\Theta(n)}$ for this range of $q$, as claimed.
\newline	
\noindent\textbf{Case 5: $q=\binom{s}{2}+ s-2$.} This case was proved by Mubayi and Terry~\cite[Theorem 8(b)]{mt2}, so we omit the proof here.	
\newline
\noindent\textbf{Case 6: $q=\binom{s}{2}+ s-1$.} This is the case requiring most work. We begin by observing that if $H$ is a graph  on $n$ vertices with girth at least $s+1$, then we may define a multigraph $G_H$ on the same vertex set by letting the multiplicity of an edge $uv$ be equal to $2$ if $uv\in E(H)$ and otherwise let its multiplicity be $1$. Clearly, $P(G_H)= 2^{e(H)}$. Furthermore, every $s$-set of vertices in $H$ contains at most $s-1$ edges, by the lower bound on the girth of $H$, so every $s$-set of vertices in $G_H$ spans at most $\binom{s}{2}+s-1$ edges (counting multiplicities). Thus $G_H\in \mathcal{F}(n,s, \binom{s}{2}+s-1)$ and \[\mathrm{ex}_{\Pi}(n, s, \binom{s}{2}+s-1)\geq \max_{H: \ \mathrm{girth}(G)\geq s+1, \ v(H)=n}P(G_H) = 2^{\mathrm{ex}\left(n, \left\{C_3, C_4, \ldots, C_s\right\}\right)},\]
as required. 

We now turn our attention to the upper bound. We shall prove it by induction on $n$. 
For the base case $n=s$, the upper bound 
\[\mathrm{ex}_{\Pi}(s, s, \binom{s}{2}+s-1)\leq 2^{s-1}= 2^{\mathrm{ex}\left(s, \left\{C_3, C_4, \ldots, C_s\right\}\right)}\]
follows from the integral version of the AM-GM inequality, Proposition~\ref{prop: integral AM-GM}(i).

Suppose  we have proved  that $ \mathrm{ex}_{\Pi}(n, s, \binom{s}{2}+s-1)\leq 2^{\mathrm{ex}\left(n, \left\{C_3, C_4, \ldots, C_s\right\}\right)}$ for all $n$ such that  $s\leq n\leq N$, for some $N\geq s$.  Consider a product-extremal multigraph $G\in \mathcal{F}(N+1,s,\binom{s}{2}+s-1)$.   By the maximality of $P(G)$, we may assume that every edge in $G$ has multiplicity at least $1$.

If $G \in \bigcap_{2\leq s'\leq s}\mathcal{F}(N+1, s', \binom{s'}{2}+s'-1)$, then we have that every edge of $G$ has multiplicity either $1$ or $2$ (since $G\in \mathcal{F}(N+1, 2, 2)$), and that the graph induced by the edges with multiplicity exactly $2$ has girth at least $s+1$ (since $G\in \mathcal{F}(N+1, s', \binom{s'}{2}+s'-1)$ for all $s'$ so that $3\leq s'\leq s$). Thus $P(G)\leq 2^{\mathrm{ex}\left(N+1, \left\{C_3, C_4, \ldots C_s\right\}\right)}$, and we are done.

 We may  therefore assume that there exists  $s'\in \mathbb N$ where $2\leq s'\leq s-1$ such that $G\notin \mathcal{F}(N+1, s', \binom{s'}{2}+s'-1)$. Let $s_0$ be the largest such $s'$, and let $X$ be a set of $s_0$ vertices in $G$ spanning at least $\binom{s_0}{2}+s_0$ edges. By the maximality of $s_0$, it follows that $e(G[X])$ is exactly equal to $\binom{s_0}{2}+s_0$. For the same reason, all edges from $X$ to $V(G)\setminus X$ must have multiplicity $1$: if $y\in V(G)\setminus X$ sends an edge of multiplicity at least $2$ to $X$, then $X\cup\{y\}$ is an $(s_0+1)$-set spanning at least $\binom{s_0+1}{2}+s_0+1$ edges, contradicting the maximality of $s_0$. In particular we have
\begin{align}\label{sparse case eq: bound on P(G)}
P(G)=P(G[X])\cdot P(G[V(G)\setminus X]).
\end{align}
Further, the integral AM-GM inequality (Proposition~\ref{prop: integral AM-GM}(i)) implies that $P(G[X])\leq 2^{s_0}$. We now turn our attention to $G[V(G)\setminus X]$. Observe that for all integers $s\geq 2$, $t\in \mathbb{N}$ and $n\geq s+t$,
\begin{align}\label{eq: lb on growth rate of ex-girth>s}
\mathrm{ex}\left(n, \left\{C_3, C_4, \ldots, C_s\right\}\right)\geq t + \mathrm{ex}\left(n-t, \left\{C_3, C_4, \ldots, C_s\right\}\right).
\end{align}
This inequality follows from the fact that taking any extremal $ \left\{C_3, C_4, \ldots, C_s\right\}$-free graph on $n-t$ vertices, selecting an arbitrary vertex and adding $t$ pendant edges to it yields a $ \left\{C_3, C_4, \ldots, C_s\right\}$-free graph on $n$ vertices with $t+\mathrm{ex}\left(n-t, \left\{C_3, C_4, \ldots, C_s\right\}\right)$ edges.

If $\vert V(G)\setminus X\vert=N+1-s_0 \geq s$, then by \eqref{sparse case eq: bound on P(G)}, \eqref{eq: lb on growth rate of ex-girth>s}, our upper bound on $P(G[X])$ and our inductive hypothesis we have 
\begin{align*}
P(G)&=P(G[X])\cdot P(G[V(G)\setminus X])\leq 2^{s_0} 2^{\mathrm{ex}\left(N+1-s_0, \left\{C_3,C_4, \ldots C_s\right\}\right)}\leq 2^{\mathrm{ex}\left(N+1, \left\{C_3,C_4, \ldots, C_s\right\}\right)},
\end{align*}
as desired.  We may thus assume $\vert V(G)\setminus X\vert=N+1-s_0 =s-s_1$ for some $s_1\geq 1$. 

We will show that  $e(G[V(G)\setminus X])\leq \binom{s-s_1}{2}+s-s_1+1$.  Suppose this is not the case. Then the following holds:
\begin{claim}\label{claim: inside  case 6 of sparse theorem}
	For every $i\in \mathbb N$ with $2\leq i \leq  s_0$, there is an $i$-set $X_i\subseteq X$ such that $X_i$ spans at least $\binom{i}{2}+i-1$ edges.
\end{claim}
\begin{proof}
 Indeed, $X=X_{s_0}$ is such a set for $i=s_0$.  Suppose now that for some $i \geq 3$ we have an $i$-subset $X_i$ of $X$ spanning at least $\binom{i}{2}+i-1$ vertices.  Consider $\sum_{x \in X_{i}} \left (e(G[X_{i}]) - d_{i}(x) \right )$, where $d_{i}(x)$ is the degree of $x$ in $G[X_{i}]$.  We have that this sum is at least $\left(i-2\right)\left({i \choose 2} + \left(i-1\right) \right)$.  Thus, by averaging there exists a vertex $x_{i} \in X_{i}$ such that $e(G[X_{i}]) - d_{i}(x_{i}) \geq {i-1  \choose 2} + (i-1) - 2 + \frac{2}{i}$.  As the quantity $e(G[X_{i}]) - d_{i}(x)$ must be an integer, we have that removing $x_{i}$ from $X_{i}$ leaves us with an $(i-1)$-set $X_{i-1}$ such that $e(G[X_{i-1}]) \geqslant {i-1  \choose 2} + (i-1) - 1$, as required.	
\end{proof}
Now $s<N+1 = s_0 +s-s_1$. In particular it follows that $s_0 >s_1$. By Claim~\ref{claim: inside  case 6 of sparse theorem}, it follows that $X$ contains an $s_1$-subset $X_{s_1}$ spanning at least $\binom{s_1}{2}+s_1-1$ edges. Then the union $X_{s_1}\cup \left(V(G)\setminus X\right)$ is a set of precisely $s$ vertices spanning at least $\binom{s_1}{2}+s_1-1 +s_1(s-s_1)+ \binom{s-s_1}{2}+(s-s_1)+1= \binom{s}{2}+s$ edges, contradicting the fact that $G\in \mathcal{F}(N+1, s, \binom{s}{2}+s-1)$.

It must therefore be the case that $e(G[V(G)\setminus X])\leq \binom{s-s_1}{2}+s-s_1$. By the integral AM-GM inequality (Proposition~\ref{prop: integral AM-GM}(i)), it follows that 
\[P(G[V(G)\setminus X]) \leq 2^{s-s_1}.\]
Substituting in $s-s_1 = N+1-s_{0}$ and combining the bound above with \eqref{sparse case eq: bound on P(G)} and our earlier upper bound of $2^{s_0}$ on  $P(G[X])$, we get
\begin{align*}
P(G)&=P(G[X])\cdot P(G[V(G)\setminus X])\leq 2^{s_0}2^{N+1-s_0}=2^{e(C_{N+1})}\leq 2^{\mathrm{ex}\left(N+1, \left\{C_3, C_4, \ldots, C_s\right\}\right)},
\end{align*}
as desired. Thus we have proved our upper bound for $n=N+1$. The claim follows by induction on $n$.

\noindent\textbf{Case 7: $\binom{s}{2}+ s-1<q <\binom{s}{2}+\Bigl\lfloor\frac{s^2}{4}\Bigr\rfloor$.} 	By maximality, all edges in $G$ have multiplicity at least $1$. We consider the subgraph $G^{\geq 2}$ of edges of multiplicity at least $2$. 
Consider any $\varepsilon >0$. Suppose $G$ contains at least $\varepsilon n^2$ edges. Then provided $n$ is sufficiently large, the Erd{\H o}s--Stone theorem tells us that $G^{\geq 2}$ contains a complete balanced bipartite graph on $s$ vertices. The induced subgraph of $G$ corresponding to these $s$ vertices then spans at least $\binom{s}{2}+ \lfloor \frac{s^2}{4}\rfloor$ edges, a contradiction. Thus we must have have that $G^{\geq 2}$ contains $o(n^2)$ edges. Since all edges in $G$ can have multiplicity at most $\frac{s^2}{4}$ +1, it follows that $P(G)\leq (\frac{s^2}{4} +1)^{e(G^{\geq 2})}=2^{o(n^2)}$, as claimed.
\newline		
\noindent\textbf{Case 8: $\binom{s}{2}+\Bigl\lfloor\frac{s^2}{4}\Bigr\rfloor \leq q  \leq 2{\binom{s}{2}}$.} Any multigraph $G'$ from $T^{e}_{2,0}(1,n)$ satisfies $P(G')=2^{\lfloor \frac{n^2}{4}\rfloor}=2^{\Omega(n^2)}$ and has the property that every $s$-set of vertices spans at most $\binom{s}{2}+\Bigl\lfloor\frac{s^2}{4}\Bigr\rfloor$ edges (counting multiplicities). This provides us with the claimed lower bound on $\mathrm{ex}_{\Pi}(n,s,q)$. For the upper bound, by averaging over all $s$-sets we see that the average edge multiplicity in $G$ is at most $2$; so by the AM-GM inequality $P(G)\leq 2^{\binom{n}{2}}$. Thus $\mathrm{ex}_{\Pi}(n,s,q)=2^{\Theta(n^2)}$ for this range of $q$.

\end{proof}

\section{Iterated constructions}\label{construct}
As the reader may have realised, the quantity $\Sigma_{r,d}(a,s)$ often jumps up by more than $1$ when we decrease $d$ or increase $r$. Thus while Conjecture~\ref{conjecture: entropy densities} covers many cases of Problem~\ref{problem: Mubayi--Terry}, there are intervals of $(s,q)$ left uncovered. In this section, we provide a generalisation of Construction~\ref{construction: lower bound} that partly address the question of what may (conjecturally) happen to $\mathrm{ex}_{\Pi}(n,s,q)$ for $(s,q)$ in these intervals. The idea behind our generalised construction is quite simple: we take a graph $G$ from $\mathcal{T}_{r,d}(a,n)$, and `iterate' inside the special part of the canonical partition $V_0=V_0(G)$, by replacing $G[V_0]$ by a graph $G'$ from $\mathcal{T}_{r',d'}(a-d, \vert V_0\vert)$, for some suitable choice of $r',d'$. This will give a slight boost to the product of edge multiplicities inside $V_0$, and we may repeat the operation further by now replacing the special part $V_0'=V_0(G')$ in $G'$'s canonical partition by a new graph $G''$ from $\mathcal{T}_{r'',d''}(a-d-d',\vert V_0'\vert)$, and so on.  We give a formal definition of our construction and an example below.
\begin{construction}[Iterated generalised Tur\'an multigraphs]\label{iterated construction}
	An \emph{admissible pair} $(\mathbf{r}, \mathbf{a})$ is a pair of $k$-dimensional vectors for some $k\in \mathbb{N}$ with strictly positive integer entries  $\mathbf{r}=(r_1, r_2, \ldots, r_k)$ and $\mathbf{a}=(a_1, a_2, \ldots, a_k)$  satisfying 
	\begin{align*}
	a_1>a_2>\ldots > a_{k}.
	\end{align*}
	Given an admissible pair $(\mathbf{r}, \mathbf{a})$, let $\mathcal{T}_{\mathbf{r}}(\mathbf{a}, n)$ denote the collection of multigraphs $G$ on $[n]$ for which $[n]$ can be partitioned into $R:=\sum_{j=1}^{k}r_j$ parts $V_1, V_2,\ldots ,V_R$ such that the following hold:
	\begin{enumerate}[(i)]
		\item for all $\ell$ such that $\sum_{j=1}^{i-1}r_j < \ell \leq \sum_{j=1}^{i}r_j$, all edges internal to $V_{\ell}$  have multiplicity $a_i$;
		\item for all $\ell_1<\ell_2$ such that $\sum_{j=1}^{i-1} r_j<\ell_1 \leq  \sum_{j=1}^{i}r_j$, all edges from $V_{\ell_1}$ to $V_{\ell_2}$ have multiplicity $a_i+1$.
	\end{enumerate}
\end{construction}
\begin{remark}
	Observe that this construction does indeed generalise Construction~\ref{construction: lower bound}, with $\mathcal{T}_{r,d}(a,n)= \mathcal{T}_{(r-1,1)}\bigr((a,a-d),n\bigl)$. It also contains new and very much different constructions, as can be seen by considering  for example the admissible pair $(\mathbf{r}, \mathbf{a})=((1,2,1),(a, a-2, a-3))$ for $a\in \mathbb{Z}_{\geq 4}$ illustrated in Figure~\ref{figure: iterated example} below.
\end{remark}
\begin{figure}[ht]
	\centering
	\includegraphics[scale=0.8]{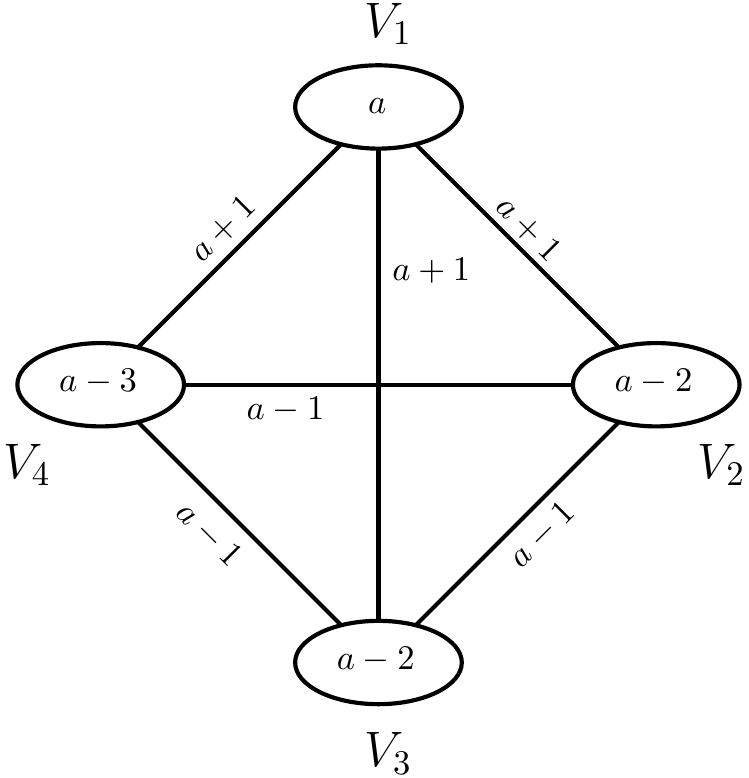}
	\caption{An example of the structure of multigraphs in $\mathcal {T}_{(1,2,1)}((a,a-2,a-3),n)$.}
	\label{figure: iterated example}
\end{figure}

Given an admissible pair $(\mathbf{r}, \mathbf{a})$, let $\Sigma_{\mathbf{r}}(\mathbf{a}, n)$ and $\Pi_{\mathbf{r}}(\mathbf{a},n)$ denote the maximum over all $G\in \mathcal{T}_{\mathbf{r}}(\mathbf{a}, n)$ of $e(G)$ and $P(G)$ respectively. Clearly, we have that for all $s$ and $q=\Sigma_{\mathbf{r}}(\mathbf{a}, s)$,
\begin{align}\label{eq: lower bound on extremal product}
\mathrm{ex}_{\Pi}(n, s, q)\geq \Pi_{\mathbf{r}}(\mathbf{a},n).
\end{align}
We say that an admissible pair $(\mathbf{r}, \mathbf{a})$ is \emph{$s$-dominant} if for all admissible pairs $(\mathbf{r'}, \mathbf{a'})$, either $\Sigma_{\mathbf{r'}}(\mathbf{a'},s)> \Sigma_{\mathbf{r}}(\mathbf{a},s)$ or $\Pi_{\mathbf{r'}}(\mathbf{a'},n)\leq \Pi_{\mathbf{r}}(\mathbf{a}, n)$ for all $n$ sufficiently large; in other words, a pair $(\mathbf{r}, \mathbf{a})$ is $s$-dominant if it gives us the best lower bound in~\eqref{eq: lower bound on extremal product} over all admissible pairs. So for example $((1,2), (a, a-1))$ is not $5$-dominant, because the pair $((2), (a))$ satisfies $\Sigma_{(2)}((a), 5)=\Sigma_{(1,2)}((a,a-1),5)=\binom{5}{2}a+6$ and $\Pi_{(2)}((a), n)\gg \Pi_{(1,2)}((a,a-1),n)$. If one believes Conjecture~\ref{conjecture: entropy densities}, it is plausible that the answer to the following question is affirmative.
\begin{question}\label{question: iterated construction}
 Is it the case that for all $s$-dominant admissible pairs $(\mathbf{r}, \mathbf{a})$,  there exists a natural number $n_0$ such that for all $n\geq n_0$, we have equality in~\eqref{eq: lower bound on extremal product}?
\end{question}

We now gather together some elementary facts about Construction~\ref{iterated construction} and the associated extremal quantity $\Pi_{\mathbf{r}}(\mathbf{a},n)$. Given an admissible pair $(\mathbf{r}, \mathbf{a})$ of $k$-dimensional vectors, set $R:=\sum_{i=1}^kr_i$.  We say that an $R$-dimensional vector $\mathbf{x}$ with $x_i\geq 0$ and $\sum_{i=1}^R x_i=1$ is a \emph{product-optimal weighting} (or \emph{p.o.w.}) if there exists a sequence of product-maximising multigraphs $G_n\in \mathcal{T}_{\mathbf{r}}(\mathbf{a}, n)$ such that the associated sequence of partitions $\left(\sqcup_{i=1}^R V_i\right)_n$ satisfies:
\begin{align*}
\lim_{n\rightarrow \infty}  \frac{\vert V_i\vert }{n}=x_i && \textrm{for all }i\in [R]. \end{align*}
We also let $\pi_{\mathbf{r}}(\mathbf{a})$ denote the entropy density of the family $\mathcal{T}_{\mathbf{r}}(\mathbf{a},n)$,\footnote{The sequence $\frac{\log \Pi_{\mathbf{r}}(\mathbf{a},n)}{\binom{n}{2}}$ is non-increasing by a classical averaging argument of  Katona, Nemetz and Simonovits~\cite{KatonaNemetzSimonovits}, and it lies within the interval $[\log a_1, \log (a_1+1))$; it thus converges to a limit as claimed, and the entropy density is thus well-defined.}
\[\pi_{\mathbf{r}}(\mathbf{a}) :=\lim_{n\rightarrow \infty}\frac{\log \Pi_{\mathbf{r}}(\mathbf{a},n)}{\binom{n}{2}}.\]

\begin{proposition}\label{prop: uniqueness of pow}
	Let $(\mathbf{r}, \mathbf{a})$ be an admissible pair. Then there exists a unique p.o.w.\ $\mathbf{x}$ for $(\mathbf{r}, \mathbf{a})$.
\end{proposition}
\begin{proof}
Let $(\mathbf{r}, \mathbf{a})$ be an admissible pair of $k$-dimensional vectors with $R=\sum_{i=1}^k r_i$. A multigraph $G\in \mathcal{T}_{\mathbf{r}}(\mathbf{a}, n)$ with canonical partition $\sqcup_{i=1}^RV_i$ may be viewed as consisting of $k$ \emph{layers} $L_j$, $1\leq j \leq k$, with layer $L_j$ corresponding to $\sqcup\{V_k: \ \sum_{i=1}^{j-1}r_i < k \leq \sum_{i=1}^{j}r_j\}$.

Suppose that $P(G)=\Pi_{\mathbf{r}}(\mathbf{a},n)$. Then we may assume that its canonical partition satisfies $\vert V_{\ell_1}\vert \geq \vert V_{\ell_2}\vert$ for all $\ell_1<\ell_2$; indeed if $\vert V_{\ell_2}\vert > \vert V_{\ell_1}\vert$ then shifting a vertex from $V_{\ell_2}$ to $V_{\ell_1}$ cannot  decrease the value of $P(G)$. A simple optimisation (alternatively Tur\'an's theorem) then tells us that within each of the layers $L_j$, the canonical partition of $G$ is balanced: if  $\sum_{j=1}^ir_j < \ell_1< \ell_2 \leq \sum_{j=1}^{i+1}r_j$, then  $0\le \vert V_{\ell_1}\vert -\vert V_{\ell_2}\vert\leq 1$.

	We can now complete the proof of the proposition by induction on the number $k$ of layers. For $k=1$, we are done by Tur\'an's theorem. Suppose we have proved uniqueness of the p.o.w. for all admissible pairs of $k$-dimensional vectors with $k\leq K$. Consider a pair of admissible $(K+1)$-dimensional vectors $(\mathbf{r}, \mathbf{a})$. Given that, as shown above, a product-maximising construction must be balanced inside each of its layers, all we have to do is to decide the proportion $1-p$ of the total number of vertices we have to put on the first $r_1$ sets $V_1, \ldots ,V_{r_1}$. To maximise the value of the product we must maximise 
	\begin{align}\label{eq: uniqueness of pow}
(1-p)^2 \log  \left((a_1)^{\frac{1}{r_1}}(a_1+1)^{\frac{r_1-1}{r_1}}\right) + 2(1-p)p\log(a_1+1) + p^2\pi_{\mathbf{r'}}(\mathbf{a'}),\end{align}
where $(\mathbf{r}', \mathbf{a}')$ is the admissible pair obtained from $(\mathbf{r}, \mathbf{a})$ by removing the first coordinates from $\mathbf{r}$ and $\mathbf{a}$. Taking the derivative of the expression above with respect to $p$, we get
\begin{align*}
2\left(\frac{1}{r_1}\log \left(\frac{a_1+1}{a_1}\right)   - p \left(\log \left(a_1+1\right)+\frac{1}{r_1}\log\left(\frac{a_1+1}{a_1}\right) -\pi_{\mathbf{r'}}(\mathbf{a'}) \right)\right).
\end{align*}
Since we started with an admissible pair, $a_2<a_1$ and thus $\pi_{\mathbf{r}'}(\mathbf{a}')< \log(a_2+1)\leq \log(a_1)$, from which it follows that the function of $p$ in~\eqref{eq: uniqueness of pow} attains a unique maximum  in $[0,1]$ at $p=p_{1}$, where
\[ p_{1} = \frac{\log \left(\frac{a_1+1}{a_1}\right)}{\log\left(\frac{(a_1+1)^{r_1+1}}{a_1 \exp\left( r_1\pi_{\mathbf{r}'}(\mathbf{a}') \right) } \right)}\in \bigl(0, \frac{1}{r_1+1}\bigr)\]
 The uniqueness of the p.o.w.\ for $(\mathbf{r}, \mathbf{a})$ then follows by induction on $k$ and our observation that product-maximising constructions are balanced inside each of their layers.
\end{proof}
\begin{remark}\label{remark> positive ccoordinates}
	The proof above also establishes that all coordinates of a p.o.w.\ are strictly positive.
\end{remark}
	The proof of Proposition~\ref{prop: uniqueness of pow} does more than just show the uniqueness of the p.o.w.: it gives an explicit recursive algorithm for computing both p.o.w.\ and entropy densities $\pi_{\mathbf{r}}(\mathbf{a})$: 
\begin{corollary} \label{corollary: computing the entropy densities, iterated case} The following hold:
\begin{enumerate}[(i)]
	\item for every $r,a\in \mathbb{N}$, $\pi_{(r)}((a))= \log(a)+ \frac{r-1}{r}\log\left(\frac{a+1}{a} \right)$;
	\item for every admissible pair $(\mathbf{r}, \mathbf{a})=((r_1, \ldots, r_k), (a_1, \ldots, a_k))$ of $k$-dimensional vectors with $k\geq 2$, 
	\[ \pi_{\mathbf{r}}(\mathbf{a})=\log (a_1) +  \frac{\log \left(\frac{(a_1+1)^{r_1}}{a_1 \exp\left((r_1-1)\pi_{\mathbf{r'}}(\mathbf{a'})\right)  }\right)} {\log \left(\frac{(a_1+1)^{r_1+1}}{a_1 \exp\left(r_1\pi_{\mathbf{r'}}(\mathbf{a'})\right)  }\right)} \log\left(\frac{a_1+1}{a_1} \right),\]
	where $(\mathbf{r'}, \mathbf{a'})$ is the admissible pair of $(k-1)$-dimensional vectors obtained by  removing the first coordinates in $\mathbf{r}$ and $\mathbf{a}$, i.e., $(\mathbf{r'}, \mathbf{a'}):=((r_2, \ldots, r_k), (a_2, \ldots, a_k))$. 
\end{enumerate}
\end{corollary}	
	\begin{proof}
		Substitute the value of $p_1$ computed in the proof of Proposition~\ref{prop: uniqueness of pow} into~\eqref{eq: uniqueness of pow} and simplify.
	\end{proof}
We define an order on admissible pairs as follows. First, set $(r,a) \prec (r',a')$ if either $a<a'$ or $a=a'$ and $r<r'$. 	Let $(\mathbf{r}, \mathbf{a})$ and $(\mathbf{r'}, \mathbf{a'})$ be admissible pairs of $k$-dimensional and $k'$-dimensional vectors respectively. We set $(\mathbf{r}, \mathbf{a}) \prec (\mathbf{r'}, \mathbf{a'})$ if either there exists some $i\leq k$ such that  for all $j<i$, $(r_j, a_j)=(r'_j, a'_j)$ and $(r_i, a_i) \prec (r'_i, a'_i)$, or $k<k'$ and for all $i\leq k$, $(r_i, a_i)=(r'_i, a'_i)$. We observe that this gives a linear ordering of the set of admissible pairs.
\begin{proposition}
	Let $(\mathbf{r}, \mathbf{a})$ and $(\mathbf{r'}, \mathbf{a'})$  be admissible pairs with $(\mathbf{r}, \mathbf{a}) \prec (\mathbf{r'}, \mathbf{a'})$. Then
	\[ \log a_1\leq  \pi_{\mathbf{r}} (\mathbf{a})< \pi_{\mathbf{r'}} (\mathbf{a'})<\log (a'_1+1).\]
\end{proposition}	
\begin{proof}
	This is an easy exercise in optimisation. Clearly, for any admissible pair of $k$-dimensional vectors $(\mathbf{r}, \mathbf{a})$, we have $\log(a_1)\leq \pi_{\mathbf{r}}(\mathbf{a}) < \log (a_1+1)$. Furthermore,  it is also clear that $\pi_{\mathbf{r}}(\mathbf{a})$ increases strictly if one strictly increases one of the coordinates of $\mathbf{r}$ or $\mathbf{a}$.

	Finally, note that if $(\mathbf{r'}, \mathbf{a'})$ is an admissible pair of $k'$-dimensional vectors for some $k'\geq 2$, then the admissible pair $(\mathbf{r}, \mathbf{a})$  obtained from $(\mathbf{r'}, \mathbf{a'})$ by removing the last coordinates ${r'}_{k'}$ and ${a'}_{k'}$ of $\mathbf{r'}$ and $\mathbf{a'}$, then $\pi_{\mathbf{r}} (\mathbf{a})< \pi_{\mathbf{r}'} (\mathbf{a}')$.  (This follows from Remark~\ref{remark> positive ccoordinates} showing that all coordinates in a p.o.w. are strictly positive, the maximality of the entropy densities and the fact that  in this case $\mathcal{T}_{\mathbf{r}}(\mathbf{a})\subseteq \mathcal{T}_{\mathbf{r'}}(\mathbf{a'})$.)
	Combining these observations with our definition of the linear ordering $\prec$ yields the claimed result.
\end{proof}
Computing the precise value of $\Sigma_{\mathbf{r}}(\mathbf{a}, s)$ for some fixed $s$ seems arduous in general, relying as it does on optimising an iterated construction. So we content ourselves with providing an analogue of Proposition~\ref{prop: sum-extremal subgraphs, threshold for 2 vertices in extremal part} for `$2$-layer' constructions.
\begin{proposition}\label{prop: 2-layer, sum-extremal}
Let $(\mathbf{r}, \mathbf{a})=((r_1,r_2), (a_1,a_2))$ be an admissible pair. Let $G\in \mathcal{T}_{\mathbf{r}}(\mathbf{a}, s)$ with $e(G)=\Sigma_{\mathbf{r}}(\mathbf{a},s)$, and let $\sqcup_{i=1}^{r_1+r_2}V_i$ be  a canonical partition of $V(G)$. Then the following hold:
\begin{enumerate}
	\item if $s\geq r_1r_2(a_1-a_2)+r_1+r_2+1$, then $\sum_{i=r_1+1}^{r_2} \vert V_i\vert \geq r_2+1$ can be achieved;
	\item if $s\leq r_1r_2(a_1-a_2)+r_1+r_2$, then $\sum_{i=r_1+1}^{r_2} \vert V_i\vert \leq r_2$.
\end{enumerate}
\end{proposition}
\begin{proof}
This is an easy adaptation of the argument from Proposition~\ref{prop: sum-extremal subgraphs, threshold for 2 vertices in extremal part}. For the first part, it is enough to show the result when $s=r_1r_2(a_1-a_2)+r_1+r_2+1$. By Tur\'an's theorem, we know that the sizes of the sets $V_i$ are balanced within each of the layers $L_1:=\sqcup_{i=1}^{r_1}V_i$ and $L_2:=\sqcup_{i=r_1+1}^{r_2}V_i$. Suppose that $\vert L_2\vert =r$, for some $r\leq r_2$. Switch a vertex from the largest part in $V_1$ to the smallest part in $L_2$. If $r<r_2$, this changes $e(G)$ by
\[\left(\left\lceil \frac{s-r}{r_1}\right\rceil -1\right)- r(a_1-a_2)\geq  \left\lceil \frac{2}{r_1}\right\rceil \geq 1>0.\]
Thus by maximality of $e(G)$ we may assume $r=r_2$, in which case our vertex-switch changes $e(G)$ by
\[\left(\left\lceil \frac{s-r_2}{r_1}\right\rceil -1\right)- r(a_1-a_2)-1=   \left\lceil \frac{1}{r_1}\right\rceil -1  =0,\]
and thus $\vert L_2\vert =r_2+1$ can be achieved while maximising $e(G)$, as claimed.

For the second part, it is enough to show the result when $s=r_1r_2(a_1-a_2)+r_1+r_2$. Essentially the same vertex-switching argument as above then shows that $\vert L_2\vert \leq r_2$ if maximising $e(G)$. 
\end{proof}

\section{Open Problems}
Much work remains to be done on the Mubayi--Terry multigraph problem. Here we highlight a few of the specific problems left open in this paper:
\begin{itemize}
	\item tightening the bounds on $n$ in Theorem~\ref{theorem: turan}: the correct bound should probably just be $n\geq s$. In the special case where $r=s-1$, this was proved by Mubayi and Terry~\cite[Theorem 10]{mt2};
	\item tackling the special cases $(s,q)=(5, \binom{5}{2}a+4)$ for any fixed  $a\in \mathbb{Z}_{\geq 3}$ and $(s,q)=(6, \binom{6}{2}a+11)$ for any fixed $a \in \mathbb{Z}_{\geq 2}$ would probably go quite some way towards settling Conjecture~\ref{conjecture: entropy densities} in full generality (these specific cases are interesting, as they are some of the `smallest' cases open, and also because of their conjectured extremal structure, which come from the families $\mathcal{T}_{2,2}(a, n)$ and $\mathcal{T}_{3,1}(a,n)$ respectively). {\bf Remark:}\emph{ the case $(s,q)=(6, \binom{6}{2}a+11)$ of Conjecture~\ref{conjecture: entropy densities} was proved asymptotically by the second author in a recent preprint~\cite{vfr}};
	\item tackling cases of $(s,q)$ not covered by Conjecture~\ref{conjecture: entropy densities}. Among those, the following cases are of particular interest: $(s,q)=(5,24)=(5, \binom{5}{2}2+4)$ (for which constructions from $\mathcal{T}_{2,2}(2,n)$ are not available, see below), $(s,q)=(6, \binom{6}{2}a+8)$ for $a\in \mathbb{Z}_{\geq 2}$ (where we provide an iterated construction from $\mathcal{T}_{(1,2)}((a, a-1), n)$ as a potential extremal candidate) and $(s,q)=(5, \binom{5}{2}a+7)$ or $(6, \binom{6}{2}a+10)$ for $a\in \mathbb{Z}_{\geq 2}$ (where we do not have any candidate extremal example).
\end{itemize}
\smallskip
The overarching goal is of course to be able to determine at least the value of $\mathrm{ex}_{\Pi}(s,q)$
 for \emph{all} pairs $(s,q)$. This appears to be a difficult problem in general. For instance, in the case of $(s,q)=(5,24)$ alluded to above, we have $q=\Sigma_{2,2}(2,s)$ --- but all multigraphs $G$ in $\mathcal{T}_{2,2}(2,n)$ that set more than two vertices inside the part $V_0$ have $P(G)=0$.  A trivial lower bound on $\mathrm{ex}_{\Pi}(5, 24)$ is $\mathrm{ex}_{\Pi}(5,23)=2$ (as shown by Mubayi and Terry in~\cite[Theorem 8]{mt2}), which can be attained by e.g. considering the multigraph on $n$ vertices in which every edge has multiplicity $2$. A better construction is obtained by considering the following example:
\begin{construction}
Let $\sqcup_{i=1}^6 V_i$ be a balanced partition of $[n]$. Define a multigraph $H_6(n)$ on $[n]$ by
\begin{enumerate}[(i)]
	\item assigning a multiplicity of $1$ to every pair internal to some $V_i$, $1\leq i \leq 6$;
	\item assigning a multiplicity of $3$ to every pair from $V_i\times V_{i+1}$, $1\leq i \leq 6$, where we follow the convention that $V_{6+1}:=V_1$;
	\item assigning a multiplicity of $2$ to every other pair from $[n]^{(2)}$.
\end{enumerate}
\end{construction} 
\noindent It is easily checked that $H_6(n)\in \mathcal{F}(n, 5, 24)$, thereby showing $\mathrm{ex}_{\Pi}(5, 24)\geq 3^{\frac{1}{3}}2^{\frac{1}{2}}$, which is strictly greater than $\mathrm{ex}_{\Pi}(5, 23)= 2$. But we have no good intuition as to whether this construction is optimal or not, or  how to generalise it to cover other cases where $(s,q)=\Sigma_{r,d}(a,s)$ and $a\leq d$: a complete conjectural classification of the form $\mathrm{ex}_{\Pi}(s,q)$ should take for any given $(s,q)$ is still very much lacking.

\section*{Acknowledgements}
The second author is grateful to Saga Samuelsson, whose undergraduate project led him to consider the Mubayi--Terry multigraph problem. Research on this problem was conducted over a research visit by the third author to Ume{\aa} University in October 2019, whose hospitality is gratefully acknowledged. Research of the first and second authors was funded by a grant from Vetenskapsr{\aa}det, for whose financial support they are most grateful. The authors are also grateful to the referees for their helpful and careful reviews.


\appendix
\section{Proof of transcendentality conditional on Schanuel's conjecture}

\begin{conjecture}[Schanuel's Conjecture]\label{conjecture: Schanuel}
Suppose the complex numbers $z_1, z_2, \ldots, z_n$ are linearly independent over $\mathbb{Q}$. Then the field extension $\mathbb{Q}(z_1,z_2, \ldots, z_n, e^{z_1},e^{z_2}, \ldots , e^{z_n})$ has transcendence degree at least $n$ over $\mathbb{Q}$.
\end{conjecture}
\noindent Recall that given integers $a,r,d$  with $a>d\geq 0$ and $r\in \mathbb{N}$, we defined quantities
\[x_{r\star}(a,d):= \frac{\log \left((a+1)/a\right)}{\log \left((a+1)^r/(a(a-d)^{r-1})\right)} \]
and
\[\pi_{r,d}(a):= \log(a)+\left(\frac{r-2+x_{r\star}(a,d)}{r-1}\right)\log\left(\frac{a+1}{a}\right).\]
Note that this version of the definition of $\pi_{r,d}(a)$ follows by considering (\ref{eq: x_star property}).
Assuming Schanuel's conjecture, one can reduce the problem of showing $x_{r\star}(a,d)$, $\pi_{r,d}(a)$ and $e^{\pi_{r,d}(a)}$ are transcendental for $d\geq 1$ to that of showing that certain sets are linearly independent over $\mathbb{Q}$ (see e.g. Lemma~\ref{lemma: a-dgeq 2, quantities are transcendental modulo} below).  Using Mih\u{a}ilescu's 2004 proof of the Calatan conjecture~\cite{Mihailescu04}, we are able to verify the required linear independence in general and obtain the following:
\begin{proposition}\label{prop: transcendentality}
Suppose Schanuel's conjecture is true. Then for all integers $a>d\geq 1$ and $r\geq 2$, the quantities $x_{r\star}(a,d)$, $\pi_{r,d}(a)$ and $e^{\pi_{r,d}(a)}$ are all transcendental. 
\end{proposition}
\begin{proof}
We begin the proof by showing how Schanuel's conjecture may help us to reduce the problem of showing transcendentality to the problem of linear independent over $\mathbb{Q}$ in general.
\begin{lemma}\label{lemma: a-dgeq 2, quantities are transcendental modulo}
Let $a,d$ be natural numbers with $a-d\geq 2$. Suppose Schanuel's conjecture is true and that $\{\log (a+1), \log (a), \log (a-d)\}$ is a linearly independent set over $\mathbb{Q}$. Then the quantities $x_{r\star}(a,d)$, $\pi_{r,d}(a)$ are transcendental over $\mathbb{Q}$. If in addition $\{ \log(a+1), \log(a), \log(a-d), x_{r\star}(a,d)\log\left(\frac{a+1}{a}\right)\}$ is a linearly independent set over $\mathbb{Q}$, then the quantity $e^{\pi_{r,d}(a)}$ is transcendental over $\mathbb{Q}$ as well.
\end{lemma}
\begin{proof}
By Schanuel's conjecture and our assumption of linear independence over $\mathbb{Q}$ of the set $\{\log (a+1), \log (a), \log (a-d)\}$, the field extension $\mathbb{Q}(\log (a+1), \log (a), \log (a-d), a+1, a, a-d)$ (i.e. the field extension $\mathbb{Q}(\log(a+1), \log(a), \log (a-d))$) has transcendence degree at least (in fact exactly) three over $\mathbb{Q}$. In particular, for any non-zero polynomial with rational coefficients $P=P(x,y, z)$, we have $P(\log(a+1), \log(a), \log (a-d))\neq 0$.

Now $x_{r\star}(a,d)$ and $\pi_{r,d}(a)$ are non-zero rational functions of $\log(a+1), \log(a)$ and $\log (a-d)$ with rational coefficients. Explicitly, we have $x_{r\star}=F(\log(a+1), \log(a), \log(a-d))$ where $F(x,y,z)= \frac{x-y}{rx-y-(r-1)z}$, and $\pi_{r,d}(a)=G(\log(a+1), \log(a), \log (a-d)$ where $G(x,y,z)=y +\left(\frac{r-2+F(x,y,z)}{r-1}\right)(x-y)$. Since $r>1$, neither $F$ nor $G$ are constant polynomials. Thus if there exists some non-zero polynomial $Q=Q(x)$ with rational coefficients such that $Q(x_{r\star}(a,d))=0$, then there exists a non-zero polynomial $P=P(x,y,z)$ with rational coefficients such that $P(\log(a+1), \log(a), \log (a-d))= 0$, a contradiction. It follows that $x_{r\star}(a,d)$ is transcendental over $\mathbb{Q}$. The transcendentality of $\pi_{r,d}(a)$ follows similarly.

Finally if $\{ \log(a+1), \log(a), \log(a-d), x_{r\star}(a,d)\log\left(\frac{a+1}{a}\right)\}$ is a linearly independent set over $\mathbb{Q}$, then so is  $\{\log (a+1), \log (a), \log (a-d), \pi_{r,d}(a)\}$. Schanuel's conjecture then implies that the field extension  $\mathbb{Q}\left( \log(a+1), \log(a), \log(a-d), \pi_{r,d}(a), a+1, a, a-d, e^{\pi_{r,d}(a)}\right)$ (i.e. the field extension $\mathbb{Q}\left( \log(a+1), \log(a), \log(a-d), \pi_{r,d}(a), e^{\pi_{r,d}(a)}\right)$) has transcendence degree at least four over $\mathbb{Q}$. Since $\pi_{r,d}(a)=G(\log(a+1), \log(a), \log(a-d))$ and $G$ is a rational function with rational coefficients, it follows that the set $\{\log(a+1), \log(a), \log(a-d), \pi_{r,d}(a)\}$ is not algebraically independent over $\mathbb{Q}$. Our bound on the transcendence degree of  $\mathbb{Q}\left( \log(a+1), \log(a), \log(a-d), \pi_{r,d}(a), e^{\pi_{r,d}(a)}\right)$ then implies that $e^{\pi_{r,d}(a)}$ must be transcendental.

\end{proof}
We now apply Lemma~\ref{lemma: a-dgeq 2, quantities are transcendental modulo} to deal with the most general case of Proposition~\ref{prop: transcendentality}.
\begin{lemma}\label{lemma: transcendentality, general case}
		Assuming Schanuel's conjecture, for all integers $a,d$ with $a\geq d+2\geq 3$ and such that neither $a+1$ nor $a$ is a rational power of $a-d$,  the quantities $x_{r\star}(a,d)$, $\pi_{r,d}(a)$ and $e^{\pi_{r,d}(a)}$ are transcendental for all integers $r\in \mathbb{Z}_{\geq 2}$.
\end{lemma}
\begin{proof}
By Lemma~\ref{lemma: a-dgeq 2, quantities are transcendental modulo}, it is enough to show that $\{\log (a+1), \log (a), \log (a-d), x_{r\star}(a,d)\log\left(\frac{a+1}{a}\right) \}$ is a linearly independent set over $\mathbb{Q}$.

We shall do this in two stages, by first proving 	$\{\log (a+1), \log (a), \log (a-d)\}$ is a linearly independent set over $\mathbb{Q}$. Indeed, suppose this was not the case. Then by clearing denominators of a rational linear combination of the these three elements, we obtain that there exist integers $\ell, m, n \in \mathbb{Z}$, not all zero, such that
\[\ell \log(a+1)+ m\log (a)+ n \log (a-d)=0,\]
i.e 
\[(a+1)^{\ell}a^m (a-d)^n=1.\]	
Since the  integers $(a+1)$ and $a$ are strictly greater than $2$ and coprime, it follows that $n\neq 0$. Since neither $a+1$ nor $a$ is a rational power of $a-d$, it follows that both $\ell$ and $m$ must be non-zero. Further, by the coprimality of $a+1$ and $a$, we must have that $\ell$ and $m$ have the same sign and that $n$ has the opposite sign, i.e. we may assume without loss of generality that $\ell,m>0$ and $n<0$, and that we have 
\begin{align}\label{eq: indep}
(a+1)^{\ell}a^m=(a-d)^{-n}.
\end{align} We claim that $a+1$ and $a$ are both integer powers. Indeed, let $\prod_{i= 1}^{s} p_i^{\alpha_i}$ be the prime factorisation of $a+1$. Let $\alpha:=\mathrm{gcd}\{\alpha_i: \ 1\leq i\leq s\}$. Suppose $\alpha=1$. By \eqref{eq: indep} and the coprimality of $(a+1)$ and $a$, we have that for $1\leq i\leq s$, $p_i$ is a factor of $(a-d)$ with some multiplicity $\beta_i>1$ satisfying $\beta_i = -(\alpha_i \ell)/n$. Now~\eqref{eq: indep} implies that $\ell<-n$, and thus that $-\ell/n$ can be written as $\ell'/n'$ where $\ell'$ and $n'$ are coprime and $n'>1$. In particular since $\beta_i$ is an integer for every $i$, this implies that $n'$ is a factor of $\alpha_i$ for every $i$, contradicting our assumption that $\alpha=1$. It thus follows that $\alpha>1$, and  $(a+1)$ is the $\alpha$-th power of some natural number strictly greater than $1$. By exactly the same argument, we have that $a$ is the $\beta$-th power of some natural number strictly greater than $1$, for some $\beta >1$. Thus $a+1$ and $a$ are integer powers differing by exactly $1$. It then follows from Mih\u{a}ilescu's Theorem~\cite{Mihailescu04} that $a+1=3$ and $a=2$, which contradicts our assumption on $a$. Thus $\{\log (a+1), \log (a), \log (a-d)\}$ is a linearly independent set over $\mathbb{Q}$ as claimed.

Since $e^{\log (a+1)}, e^{\log a }, e^{\log (a-d)}$ are all integers, the linear independence over $\mathbb{Q}$ we have just established together with Schanuel's conjecture implies that $\{\log(a+1), \log(a), \log(a-d)\}$ are algebraically independent over $\mathbb{Q}$. We now use this fact to establish  that $\{\log (a+1), \log (a), \log (a-d), x_{r\star}(a,d)\log\left(\frac{a+1}{a}\right) \}$ is a linearly independent set over $\mathbb{Q}$, thereby completing the proof. Note $x_{r\star}(a,d)\log\left(\frac{a+1}{a}\right)=f(\log(a+1), \log(a))/g(\log(a+1), \log(a), \log(a-d))$, where $f$ and $g$ are the linear functions
\begin{align*}
f(x,y):= (x-y)^2 &&  g(x,y,z):=rx-y-(r-1)z.
\end{align*}
Suppose for contradiction that 	$x_{r\star}(a,d)\log\left(\frac{a+1}{a}\right)$ lies in the linear span of  $\{\log (a+1), \log (a), \log (a-d)\}$ over $\mathbb{Q}$. Then, clearing denominators and multiplying out by $g(\log(a+1), \log(a), \log(a-d))$ as necessary, we obtain that there are integers $k, \ell, m,n$ with $k>0$ such that
\begin{align*}
\left(\ell \log(a+1)+m \log(a) + n\log (a-d)\right)g\left(\log (a+1),\log(a),\log(a-d)\right)=kf(\log (a+1),\log(a)). 
\end{align*}
Gathering terms, we get that $(x,y,z)=(\log(a+1), \log(a), \log(a-d))$ is a root of the degree $2$ polynomial $P$ given by
\[P(x,y,z)=x^2(k-r\ell) +y^2(k+m)+z^2(n(r-1))+xy(-2k-rm+\ell)+xz(-rn+(r-1)\ell)) +yz(m(r-1)+n).\]
Observe $P$ has integer coefficients. Since, as we showed above from Schanuel's conjecture, $\{\log(a+1),\log(a), \log(a-d)\}$ is algebraically independent over $\mathbb{Q}$, $P$ must be the zero polynomial. Going through the $x^2$ and $y^2$ coefficients in $P$, this implies that $k=r\ell=-m$, which is strictly positive by our assumption on $k$. But now consider the $xy$ coefficient in $P$, which is equal to $-2k+rk+k/r=k(r-2+1/r)>0$. This shows $P$ cannot be the zero polynomial, a contradiction. We deduce from this that $x_{r\star}(a,d)\log\left(\frac{a+1}{a}\right)$ does not lie in the linear span of  $\{\log (a+1), \log (a), \log (a-d)\}$ over $\mathbb{Q}$. This concludes the proof.
\end{proof}
The special cases where $a+1$ or $a$ is a rational powers of $a-d$ or where $a-d=1$ can be dealt with in a very similar way (though in this case we do not need to appeal to Mih\u{a}ilescu's Theorem).
\begin{lemma}\label{lemma: transcendentality special case}
Assuming Schanuel's conjecture, for all integers	$a,d$ with $a>d\geq 1$ and such that either one of $a+1$ or $a$ is a rational power of $a-d$,  or $a-d=1$ holds, the quantities $x_{r\star}(a,d)$, $\pi_{r,d}(a)$ and $e^{\pi_{r,d}(a)}$ are transcendental for all integers $r\in \mathbb{Z}_{\geq 2}$.
\end{lemma}
\begin{proof}
We begin by observing that for all $a\ge 2$, Schanuel's conjecture implies that $\log(a+1)$ and $\log(a)$ are algebraically independent over $\mathbb{Q}$. Indeed, since $(a+1)$ and $a$ are coprime, it is trivial to show that $\log(a+1)$ and $\log(a)$ are linearly independent over $\mathbb{Q}$. Schanuel's conjecture then implies that the field extension $\mathbb{Q}(\log(a+1), \log(a))$ has transcendence degree $2$, as claimed.

Since $a+1$ and $a$ are coprime, only one of them can be a rational power of $a-d$.  If one of them is a rational power of $(a-d)$ or if $a-d=1$, then we have that $x_{r\star}(a,d)$ and $\pi_{r,d}(a)$ are non-zero (and, since $r\geq 2$, non-constant) rational functions of $\log(a+1)$ and $\log a$ with rational coefficients. Thus by the algebraic independence of $\log(a+1)$ and $\log(a)$ established above, it follows that both $x_{r\star}(a,d)$ and $\pi_{r,d}(a)$ are transcendental.

Finally to establish the transcendentality of $e^{\pi_{r,d}(a)}$, we show that $\{\log(a+1), \log(a), x_{r\star}(a,d)\log\left(\frac{a+1}{a}\right) \}$ is linearly independent over $\mathbb{Q}$.

If $(a+1)=(a-d)^q$ for some rational $q>1$, then $x_{r\star}(a,d)\left(\log(a+1)-\log(a)\right)$ can be written as $f(\log(a+1), \log(a))/g(\log(a+1), \log(a))$ where $f$ and $g$ are the linear functions
\begin{align*}
 f(x,y):= (x-y)^2 &&  g(x,y):= (r-\frac{(r-1)}{q})x-y.
 \end{align*}
 Suppose for contradiction that $x_{r\star}(a,d)\left(\log(a+1)-\log(a)\right)$ lies in the linear span over $\mathbb{Q}$ of $\{\log(a+1), \log(a)\}$. Clearing denominators and  multiplying out by $g(\log(a+1), \log(a))$ as necessary, we obtain that there are integers $k,\ell, m$ with $k>0$ such that
 \[\left(\ell \log(a+1) + m\log(a)\right)g(\log(a+1), \log(a)) = k f(\log(a+1), \log(a)). \]
Gathering terms, we get that $(x,y)=(\log(a+1), \log(a))$ is a root of the polynomial $P$ given by 
\[P(x,y)=x^2\left(k-\ell\left(r-\frac{r-1}{q}\right)\right)	+y^2\left(k+m\right)+ xy\left(-2k + \ell -m\left(r-\frac{r-1}{q}\right)\right).\]

Now by the algebraic independence of $\{\log(a+1), \log(a)\}$ established at the start of this proof, and by the fact that $P$ has rational coefficients, we have that $P$ must be the zero polynomial. Inspecting the coefficients of $P$ for $x^2$ and $y^2$, we see that $k= \ell \left(r-\frac{r-1}{q}\right)=-m$. Since $k>0$, it follows that both $m$ and $\ell$ are non-zero. Further since $q>1$, we have that $r- \frac{r-1}{q}>1$, so that $\ell\neq -m$. But now consider the coefficient of $xy$ in $P$: given $k= \ell \left(r-\frac{r-1}{q}\right)=-m$, the quantity $-2k + \ell -m\left(r-\frac{r-1}{q}\right)$ can be rewritten as  $\left(\ell+m\right)\left(1- \left(r-\frac{r-1}{q}\right)\right)$. As we have shown $\ell\neq -m$ and as $q>1$ implies $\left(1- \left(r-\frac{r-1}{q}\right)\right)<0$, we have that the coefficient of $xy$ in $P$ is non-zero, a contradiction. It follows that $\{\log(a+1), \log(a), x_{r\star}(a,d)\left(\log(a+1)-\log(a)\right) \}$ is linearly independent over $\mathbb{Q}$, as required.

The linear independence of $\{\log(a+1), \log(a), x_{r\star}(a,d)\left(\log(a+1)-\log(a)\right) \}$ over $\mathbb{Q}$ in the cases where $a=(a-d)^q$ for some rational $q>1$ or where $a-d=1$ are obtained in a similar way, mutatis mutandis. Thus in all three cases considered in this proposition it follows that  the set $\{\log(a+1, \log(a), \pi_{r,d}(a)\}$ is linearly independent over $\mathbb{Q}$ (since $\pi_{r,d}(a)$ is a rational linear combination of $\log(a)$ and $x_{r\star}(a,d)\left(\log(a+1)-\log(a)\right)$). By Schanuel's conjecture, this linear independence over $\mathbb{Q}$ implies that the field extension $\mathbb{Q}(\log(a+1), \log(a), \pi_{r,d}(a), a+1, a, e^{\pi_{r,d}(a)})$ has transcendence degree at least three over $\mathbb{Q}$. Since $\pi_{r,d}$ can be written as a rational function of $\log(a+1)$ and $\log(a)$ with rational coefficients, the set $\{\log(a+1), \log(a), \pi_{r,d}(a)\}$ is algebraically dependent over $\mathbb{Q}$. As $a+1, a\in \mathbb{Q}$, the transcendence degree of the field extension thus implies that  $\{\log(a+1), \log(a), e^{\pi_{r,d}(a)}\}$ is algebraically independent over $\mathbb{Q}$, and in particular that $e^{\pi_{r,d}(a)}$ is transcendental as claimed.
\end{proof}
Together, Lemmas~\ref{lemma: transcendentality, general case} and~\ref{lemma: transcendentality special case} give Proposition~\ref{prop: transcendentality}.
\end{proof}
\end{document}